%% file: C-class.tex
\documentclass[12pt]{amsart}
 
 \input{DT-macros}

 \usepackage{hyperref}
 \usepackage[normalem]{ulem}
 \usepackage{marginnote}
 \usepackage{graphicx}
 \usepackage{amssymb}
 \usepackage{framed}
 \newtheorem{theorem}{Theorem}[section]
\newtheorem{lemma}[theorem]{Lemma}
\newtheorem{cor}[theorem]{Corollary}
\newtheorem{prop}[theorem]{Proposition}
\theoremstyle{defn}
\newtheorem{defn}[theorem]{Definition}
\newtheorem{example}[theorem]{Example}

\theoremstyle{remark}
\newtheorem{remark}[theorem]{Remark}
\numberwithin{equation}{section}

\newenvironment{psmallmatrix}
  {\left(\begin{smallmatrix}}
  {\end{smallmatrix}\right)}

 \newcommand\LT{\operatorname{LT}}
 \newcommand\hor{\operatorname{hor}}

 \date{\today}
 
\title{On C-class equations}
 
 \author{Andreas \v{C}ap}
 \address{Fakult\"at f\"ur Mathematik, Universit\"at Wien, Oskar-Morgenstern Platz 1, 1090 Wien, Austria}
 \email{andreas.cap@univie.ac.at}

 \author{Boris Doubrov}
 \address{Faculty of Mathematics and Mechanics, Belarusian State 
University, Nezavisimosti ave. 4, 220050, Minsk, Belarus}
 \email{doubrov@bsu.by}
 
 \author{Dennis The}
 \address{
 Department of Mathematics and Statistics, UiT The Arctic University of Norway, N-9037, Troms\o, Norway; 
 Fakult\"at f\"ur Mathematik, Universit\"at Wien, Oskar-Morgenstern Platz 1, 1090 Wien, Austria} 
 \email{dennis.the@uit.no}
 


\begin{document}
\begin{abstract}
 The concept of a C-class of differential equations goes back to E. Cartan with the upshot that generic equations in a C-class can be solved without integration.  
 While Cartan's definition was in terms of differential invariants being first integrals, 
 all results exhibiting C-classes that we are aware of are based on the fact that a canonical Cartan geometry associated to the equations in the class descends to the space
of solutions. For sufficiently low orders, these geometries belong 
to the class of parabolic geometries and the results follow from 
the general characterization of geometries descending to a twistor
space. 

In this article we answer the question of whether a canonical Cartan 
geometry descends to the space of solutions in the remaining cases of scalar ODE 
of order at least four and of systems of ODE of order at least three. 
As in the lower order cases, this is characterized by the vanishing of the
generalized Wilczynski invariants, which are defined via the linearization
at a solution. The canonical Cartan geometries (which are not parabolic
geometries) are a slight variation of those available in the literature
based on a recent general construction. All the verifications needed to 
apply this construction for the classes of ODE we study are carried out 
in the article, which thus also provides a complete alternative proof for
the existence of canonical Cartan connections associated to higher order
(systems of) ODE.

\end{abstract}

\maketitle

 \section{Introduction}

 Consider a (system of) $(n+1)$-st order ODE $\cE$ given by
 \begin{align}
 \bu^{(n+1)} = {\bf f}(t,\bu,\bu',...,\bu^{(n)}), \label{E:ODE}
 \end{align}
 where $\bu^{(k)}$ is the $k$-th derivative of $\bu = (u^1,...,u^m)$ 
 with respect to $t$.  In a short paper \cite{Car1938} in 1938, \'Elie
 Cartan defined the following notion: {\em ``A given class of ODE
   \eqref{E:ODE} will be said to be a C-class if there exists an
   infinite group (in the sense of Lie) $\fG$ transforming equations
   of the class into equations of the class and such that the
   differential invariants with respect to $\fG$ of an equation of the
   class be first integrals of the equation.''}  Here, $\fG$ is a
 prescribed (local) Lie transformation pseudogroup, e.g.\ contact
 transformations $\fC$ or point transformations $\fP$.  (Recall that
 by B\"acklund's theorem, $\fC$ is identified with $\fP$ when $m > 1$,
 but these are distinct in the case of scalar equations.)
 
 Cartan gave two examples of C-classes in the context of: (i) scalar
 3rd order ODE up to $\fC$; (ii) scalar 2nd order ODE up to
 $\fP$. These examples were based on an equivalent description as
 ``espaces g\'en\'eralis\'es''.  In modern language, one represents the equation as a submanifold $\cE$ in an appropriate jet space and endows it with a canonical Cartan geometry $(\cG \to \cE, \omega)$ (see \S \ref{S:Cartan}).  A canonical Cartan connection $\omega$ can be obtained
 using only linear algebra or differentiation via, for example,
 Cartan's method of equivalence. In particular, integration is not
 needed.  Since a Cartan connection provides a distinguished coframing
 on a principal bundle $\cG$ over the ODE $\cE$, differential invariants of the original ODE
 structure arise from the components of its curvature (and its
 covariant derivatives).  If one knows a priori that all differential
 invariants are first integrals, and there are sufficiently many
 functionally independent ones, then these can be used to solve the
 ODE.  Consequently, the utility of searching for C-classes becomes
 readily apparent: {\em generic C-class ODE can be solved without
   integration}.
 
 More recently, R.~Bryant identified in \cite{Bry1991} a C-class
 within 4th order scalar ODE (up to $\fC$), and the concept of 
 torsion-free path geometries (in the sense of Fels--torsion, see
 \cite{Fels1995}) from D.~Grossman's article \cite{Gro2000} describes
 a C-class for 2nd order systems (up to $\fC \cong \fP$).

 The foundations of the geometric study of systems of ODEs of higher order via Cartan connections were developed by N. Tanaka and were published recently as technical reports \cite{Tanaka2017a, Tanaka2017b}. In \cite[Part I, Chapter VII]{Tanaka2017a}, Tanaka gives an interpretation of higher order systems of ODEs as $G_0$ structures on filtered manifolds (cf.\ Section 2 of this paper) and constructs a scalar product to define the normalization conditions for the associated Cartan connection (cf.\ Section 3). The detailed exposition of this approach can be found in \cite{DKM1999}. In \cite{Tanaka2017b}, Tanaka also establishes the foundations for the integration of what he calls ``foliated'' Cartan connections.

 As shown in \cite{DKM1999}, both scalar ODE (of order at least 3) up
 to $\fC$ and systems of ODE (of order at least 2) up to $\fC \cong
 \fP$ admit an equivalent description via a canonical Cartan geometry\footnote{It is well known that all scalar 2nd order ODE are (locally) equivalent up to $\fC$.  Regarding them up to $\fP$ also leads to a canonical Cartan geometry, but this is exceptional from the point of view of our formulations, so will be henceforth be excluded in this article.}
 $(\cG \to \cE, \omega)$ of type $(G,P)$ for an appropriate Lie group
 $G$ and closed subgroup $P\subset G$. (We caution that the existence
 of canonical Cartan connections with respect to an arbitrary
 pseudo--group $\fG$ is not known.) In the geometric description of
 the ODE $\cE$, the solution space $\cS$ corresponds to the space of
 integral curves in $\cE$ of a certain distinguished line field
 $E\subset T\cE$, i.e.\ $\cS \cong \cE / E$. On the homogeneous model
 $G/P$ of the geometry, the space $\cS$ is given as $G/Q$ for a
 subgroup $Q\subset G$ containing $P$.  Hence, a natural question
 arises: for the given ODE, does the canonical Cartan geometry $(\cG
 \to \cE, \omega)$ of type $(G,P)$ descend to a Cartan geometry $(\cG
 \to \cS, \omega)$ of type $(G,Q)$?  If so, then all differential
 invariants of $\omega$ will be well-defined functions on $\cS$, i.e.\
 they will be constant on solutions, hence they are necessarily first
 integrals.  Thus, such ODE $\cE$ will define a C-class. On the other
 hand, this is a natural way to obtain geometric structures on the
 solution space, which are an important topic in the geometric theory
 of differential equations \cite{DT2006,Nur2009,GN2010,Kry2010,DG2012}.
 
 For the cases treated by Cartan and Grossman, the equivalent Cartan
 geometry actually falls into the class of parabolic geometries. In
 this setting, the solution space is a special instance of a twistor
 space of a parabolic geometry and the fundamental question of whether
 a parabolic geometry descends to a twistor space was studied in
 \cite{Cap2005}. It turns out that this depends only on the Cartan
 curvature, and, as observed in \cite{CS2009}, this remains true for
 arbitrary Cartan geometries. For parabolic geometries, there is a
 simpler geometric object than the Cartan curvature, which still is a
 fundamental invariant, namely the so--called {\em harmonic
   curvature}. Using the machinery of Bernstein--Gelfand--Gelfand
 sequences (BGG sequences) from \cite{CSS2001} and \cite{CD2001}, it
 was shown in \cite{Cap2005} that descending of the geometry can be
 characterized in terms of this harmonic curvature. In particular,
 this provides an alternative proof for the results by Cartan and
 Grossman.

 Our goal in this article is to extend the characterization of the
 possibility of descending the Cartan geometry to the solution space
 to higher order cases (which also recovers Bryant's result on
 C-class from \cite{Bry1991}). There are natural candidates for
 relative invariants whose vanishing should characterize this descent,
 namely the {\em generalized Wilczynski invariants}. These
 were introduced in \cite{Dou2008}, where it was shown that their
 vanishing (i.e.\ {\em Wilczynski--flatness}) implies existence of a
 certain geometric structure on the solution space.

 For concreteness, let us recall how the generalized Wilczynski
 invariants are defined.  Consider a linear ODE system:
\[
 \bu^{(n+1)}=P_n(t) \bu^{(n)}+\dots+P_0(t)\bu
\]
up to transformations $(t,\bu)\mapsto (\lambda(t),\mu(t)\bu)$, where
$\mu(t)\in \tGL(m)$. Any such system can be brought to the canonical
Laguerre--Forsyth form defined by: $P_n=0$ and $\tr(P_{n-1})=0$.

As proven by Wilczynski~\cite{Wilc1905} for scalar ODE and
generalized by Se-ashi~\cite{Seashi1988} to systems of ODE, the
following expressions become fundamental invariants for the class of
linear equations (in Laguerre--Forsyth form) and the above class of transformations:
\[
\Theta_r = \sum_{j=1}^{r-1} (-1)^j
\frac{(2r-j-1)!(n-r+j)!}{(r-j)!(j-1)!} P_{n-r+j}^{(j-1)},
\]
for $r=2,\dots, n+1$.  (Observe that $\Theta_2$ is trace--free
  and thus vanishes for scalar ODE.)

\begin{defn}\label{D:Wilc} 
  For \eqref{E:ODE}, the \emph{generalized Wilczynski invariants} $\cW_r$ for $r=2,\dots,n+1$ are defined as the invariants $\Theta_r$
  evaluated at the linearization of the system. Formally, they are
  obtained by substituting each $P_r(t)$ with the matrix
  $\left(\frac{\partial {\bf f}}{\partial \bu^{(r)}}\right)$ and
  replacing the usual derivative by the 
  total derivative 
  \[
  \tfrac{d}{dt} = \tfrac{\partial}{\partial t} + \bu^{(1)} \tfrac{\partial}{\partial \bu} + ... + \bu^{(n)} \tfrac{\partial}{\partial \bu^{(n-1)}} + {\bf f} \tfrac{\partial}{\partial \bu^{(n)}}.
  \]
 \end{defn}

Our main problem thus is to relate the Wilczynski invariants to the
curvature of the canonical Cartan geometry (which is not a parabolic
geometry for higher-order cases) and to prove that vanishing of these
invariants implies the necessary algebraic restrictions on this
curvature. Now it has been known that there is an analogue of harmonic
curvature for the Cartan geometries constructed in \cite{DKM1999}, and
the Wilczynski invariants were identified as certain components of
this harmonic curvature. However, without having the machinery of BGG
sequences at hand, it is very hard to systematically deduce restrictions on the curvature from restrictions on the harmonic curvature.  In the special case of scalar 7th order ODE, this was sorted out in \cite{DG2012} using direct computations that were not reproduced in the article.

To be able to apply BGG--like arguments, we use a small variation
of the canonical Cartan connection from \cite{DKM1999}. This is based
on the recent general construction of canonical Cartan connections
associated to filtered geometric structures in \cite{Cap2016}. This
has the advantage of a simpler characterization of the canonical
Cartan connection and of stronger uniqueness results. All
verifications needed to apply this general theory to the case of
(systems of) ODE are carried out in our article, so we obtain a
complete alternative proof of existence of canonical Cartan
connections associated to (systems of) higher order ODE.

The proof of the main result of this paper (Theorem \ref{T:Wilc-C}) is
based on arguments similar to the ones used in the recent versions of
the BGG machinery, see \S 4.9 and \S 4.10 of \cite{CSo2017}. Together
with the results of Cartan and Grossman from \cite{Car1938} and
\cite{Gro2000} (or the ones from \cite{Cap2005}), we obtain:

\begin{theorem} 
  The following families of equations and pseudogroups form C-classes:
\begin{itemize}
\item scalar ODE of order $\ge 3$ (viewed up to contact transformations) with vanishing generalized Wilczynski invariants;
\item systems of ODE of order $\ge 2$ (viewed up to point transformations) with vanishing generalized Wilczynski invariants.
\end{itemize}
\end{theorem}

Let us briefly describe the structure of the paper. In \S 2, we show
that ODE can be described as filtered geometric structures and
analyze the trivial equation to obtain the Lie groups and Lie algebras
needed for a description as a Cartan geometry. We also discuss the
space of solutions and the concept of C-class in this setting
(Definition \ref{D:C-class}). The verifications needed to apply the
constructions of canonical Cartan connections from \cite{Cap2016} are
carried out in \S 3. These are purely algebraic, partly using
finite--dimensional representation theory. In the end of the section,
we give examples of homogeneous C-class ODE.  In \S 4, we
relate the Wilczynski invariants to the curvature of the canonical
Cartan connection and prove our main result. It is worth mentioning
here that not all the filtered
geometric structures of the type we use are obtained from ODE (see
Remark \ref{R:strong-reg} and the example related to $G_2$ in \S
\ref{S:hom-C-class}). Our results continue to hold for these more
general structures, provided one uses the description of Wilczynski
invariants in Theorem \ref{T:Wilc} as a definition in this more
general setting.

 \section{Invariants and C-class via Cartan connections}

 Our results are based on an equivalent description of (systems of)
 ODE as Cartan geometries, which is a variant of the one in
 \cite{DKM1999}. This in turn is derived from an equivalent
 description as a filtered analogue of a G--structure, which we discuss
 first.
 
 \subsection{ODE as filtered $G_0$--structures}
 \label{S:geo-str}
 
 Consider the jet spaces $J^\ell = J^\ell(\bbR,\bbR^m)$, with
 projections $\pi^\ell_k : J^\ell \to J^k$ $(k < \ell)$, and standard
 adapted coordinates $(t,\bu_0,\bu_1,...,\bu_\ell)$, where $\bu_j =
 (u^1_j,...,u^m_j)$ refers to the $j$--th derivative of $\bu(t) =
 (u^1(t),...,u^m(t))$.  For $\ell \geq 1$, the (rank $m+1$) contact subbundle is $C \subset T J^\ell$, which is locally the annihilator of (the
 components of)
 \[
 \theta_0 = d\bu_0 - \bu_1 dt, \quad 
 \theta_1 = d\bu_1 - \bu_2 dt, \quad ..., \quad 
 \theta_{\ell-1} = d\bu_{\ell-1} - \bu_\ell dt.
 \]
 Its weak derived flag yields a filtration by subbundles $C =:\, C^{-1}
 \subset C^{-2} \subset ... \subset C^{-\ell-1} := TJ^\ell$, with
 $C^i$ having corank $m$ in $C^{i-1}$, and
 \[
 C^i = \tspan\{ \partial_t + \bu_1 \partial_{\bu_0} + ... + \bu_{\ell+1+i} \partial_{\bu_{\ell+i}}, \,\, \partial_{\bu_\ell}, \,\, ...,\,\, \partial_{\bu_{\ell+1+i}} \}
 \]
 for $i=-1,...,-\ell$.  The Lie bracket satisfies
 $[\Gamma(C^i),\Gamma(C^j)] \subset \Gamma(C^{i+j})$, and so
 $(J^\ell,\{ C^i \})$ becomes a filtered manifold.  In fact,
 $[\Gamma(C^i),\Gamma(C^j)] \subset \Gamma(C^{\min(i,j)-1})$, which is
 a stronger condition if $i,j \leq -2$.
 
 We will exclusively study ODE under {\em contact transformations}.
 These are diffeomorphisms $\Phi : J^\ell \to J^\ell$ such that
 $\Phi_*(C) = C$.  By B\"acklund's theorem, $\Phi$ is the prolongation
 of a contact transformation on $J^1$.  Moreover, if $m > 1$, the
 latter is the prolongation of a diffeomorphism on $J^0$, i.e.\ a {\em
   point} transformation.
 
 Suppose $n \geq 2$.  The $(n+1)$-st order ODE \eqref{E:ODE}
 corresponds to a submanifold $\cE \subset J^{n+1}$ transverse to
 $\pi^{n+1}_n$, so $\cE$ is locally diffeomorphic to $J^n$.  For $\ell
 = 1,...,n$, the contact subbundle on $J^\ell$ is preserved by contact
 transformations, and its preimage under $\pi^{n+1}_\ell|_\cE$ yields
 a subbundle $T^{\ell - n-1} \cE \subset T\cE$.  The weak derived flag
 of $D := T^{-1} \cE$ also gives rise to these same filtration
 components:
 \begin{align} \label{E:ODE-filtration}
 T^{-1}\cE \subset T^{-2}\cE \subset ... \subset T^{-n}\cE \subset T^{-n-1}\cE := T\cE.
 \end{align}
 As for jet-spaces, $(\cE,\{ T^i \cE \})$ is a filtered manifold with
 \begin{align} \label{E:strong-bracket}
  [\Gamma(T^i \cE), \Gamma(T^j \cE)] \subset \Gamma(T^{\min(i,j)-1}\cE) \subset \Gamma(T^{i+j}\cE).
 \end{align}
 Further, there are distinguished subbundles $E \subset D$ and $F^{i}
 \subset T^{i} \cE$:
 \begin{itemize}
 \item $E = \tspan\{ \frac{d}{dt} := \partial_t + \bu_1
   \partial_{\bu_0} + ... + \bu_n \partial_{\bu_{n-1}} + {\bf f}
   \partial_{\bu_n} \}$ is the annihilator of the
   pullbacks of $\theta_0,...,\theta_n$ on $J^{n+1}$ to $\cE$.
   
 \item $F^{i} = \tspan\{ \partial_{\bu_n}, ..., \partial_{\bu_{n+1+i}}
   \}$ is the (involutive) vertical bundle for
   $\pi^{n+1}_{n+i}|_\cE$.  By B\"acklund's theorem, $F:= F^{-1}
   \subset ... \subset F^{-(n-1)}$ are distinguished; $F^{-n}$ is
   distinguished for $m > 1$.  These give corresponding splittings
   $T^{i} \cE = E \op F^{i}$.
 \end{itemize}

 For $x \in \cE$, define $\fm_i(x) := T^i_x \cE / T^{i+1}_x \cE$,
 and induce a tensorial (``Levi'') bracket on $\fm(x) = \bop_{i < 0}
 \fm_i(x)$ from the Lie bracket of vector fields.  This nilpotent
 graded Lie algebra (NGLA) is the {\em symbol algebra} at $x$ of
 $(\cE,D)$, and its NGLA isomorphism type is independent of $x$,  so let $\fm$ denote a fixed NGLA with $\fm \cong \fm(x)$ for any $x$.
 Moreover, it is the same for all ODE \eqref{E:ODE}, and we describe
 it in \S \ref{S:triv-eq} below.

 All NGLA isomorphisms from $\fm$ to some $\fm(x)$ comprise the total space
 of a natural frame bundle $F_{\tgr}(\cE) \to \cE$.  This has
 structure group $\Aut_{\tgr}(\fm)$, which naturally injects into
 $\tGL(\fm_{-1}) \cong \tGL_{m+1}$ since $\fm_{-1}$ generates all of
 $\fm$, reflecting the fact that $D$ is ``bracket-generating''.  The splitting $D = E
 \op F$ is encoded via reduction to a subbundle $\cG_0 \to \cE$ with
 structure group $G_0 = \bbR^\times \times \tGL_m$ embedded as
 diagonal blocks in $\tGL_{m+1}$.

 Fixing $\fm$ and $G_0 \subset \Aut_{\tgr}(\fm)$ as above, a {\em
   filtered $G_0$-structure} consists of: 
 \begin{itemize}
 \item[(i)] a filtered manifold $(M,\{ T^i M \}_{i < 0})$ whose symbol algebras form a locally trivial bundle with model algebra $\fm$;
 \item[(ii)] a reduction of structure group of $F_{\tgr}(M) \to M$ to a
   principal $G_0$-bundle $\cG_0 \to M$.
 \end{itemize}
 Note that (i) implies that $T^{-1} M$ is of constant rank and
 bracket-generating in $TM$. As described above, any ODE $\cE$ yields a
 filtered $G_0$-structure. These are not the most general instances of
 such structures, however, since the splittings $T^{i} \cE = E \op
 F^{i}$ for $i=-2,...,-(n-1)$ (and for $i=-n$ if $m > 1$) are an
 additional input. The following discussion in fact applies to all
 filtered $G_0$--structures, and not only to those defined by (systems of)
 ODE.
 
 \subsection{The trivial ODE} \label{S:triv-eq}
 
 We exclude the cases of scalar 3rd order ODE and of (systems of) 2nd
 order ODE as these lead to parabolic geometries, which are
 structurally different.  So suppose that $n \geq 2$ and $m \geq 1$, with $(n,m) \neq\,\, (2,1)$.  Then the contact symmetry algebra $\fg$ of
 the trivial ODE $\bu^{(n+1)} = 0$ consists entirely of the
 (prolonged) point symmetries:
 \begin{align} \label{E:triv-sym}
 \partial_{u^a}, \,\, t\partial_{u^a}, \,\, ..., \,\, t^n
 \partial_{u^a}, \,\, \partial_t, \,\, t\partial_t, \,\,
 u^b\partial_{u^a}, \,\, t^2 \partial_t + ntu^a\partial_{u^a},
 \end{align}
 where $1 \leq a,b \leq m$.  Abstractly, $\fg = \fq \ltimes \fa$,
 where $\fq = \fsl_2 \times \fgl_m$ acts on the abelian ideal $\fa =
 V_n \otimes W$, with $V_n = S^n(\bbR^2)$ as an $\fsl_2$-module and $W
 = \bbR^m$.  Take a basis $\{ \sfx,\sfy \}$ on $\bbR^2$ and the standard
 $\fsl_2$-basis
 \[
 \sfX = \sfx \partial_{\sfy}, \quad
 \sfH = \sfx \partial_{\sfx} - \sfy \partial_{\sfy}, \quad
 \sfY = \sfy \partial_{\sfx}.
 \]
 On $V_n$, use the basis $\sfv^i = \frac{1}{i!} \sfx^{n-i} \sfy^i$,
 where $0 \leq i \leq n$.  Let $\{ \sfe_a \}$ and $\{ \sfe^a_b \}$ be
 the standard bases on $\bbR^m$ and $\fgl_m$, which satisfy $\sfe^a_b
 \sfe_c = \delta^a_c \sfe_b$.

 The prolongation to $J^{n+1}$ of \eqref{E:triv-sym} shows that $\fg$
 is infinitesimally transitive on $\cE \subset J^{n+1}$, with
 isotropy subalgebra $\fp \subset \fg$ at $o = \{ t=0, \bu_0 = ... =
 \bu_n = 0 \} \in \cE$ spanned by $2t\partial_t, \,\,
 u^b\partial_{u^a}, \,\, t^2 \partial_t + nt u^a\partial_{u^a}$.
 Abstractly, $\fp$ is spanned by $\sfH, \fgl_m, \sfY$.  The filtration
 \eqref{E:ODE-filtration} induces ($\fp$-invariant) filtrations on
 $\fg / \fp \cong T_o \cE$ and $\fg$:
 \[
 \fg^{-n-1} = \fg\supset \fg^{-n} \dots \supset \fg^{-1} \supset \fg^0
 = \fp \supset \fg^1\supset \{ 0 \},
 \]
 and we put $\fg^i = \{ 0 \}$ for $i \geq 2$, and $\fg^i = \fg$ for $i
 \leq -n-1$.  In particular, $\fg^{-1} / \fp \cong D_o = E_o \op F_o$,
 with $E_o \cong \bbR \sfX$ and $F_o \cong \bbR\sfy^n \otimes W$
 (modulo $\fp$), while $\fg^1 = \bbR\sfY$ is distinguished as those
 elements of $\fp$ whose bracket with $\fg^{-1}$ lies in $\fp$.
 Viewed concretely, $E_o,F_o,\fg^1$ are respectively spanned by (the
 prolongations of) $\partial_t$, $t^n \partial_{u^a}$, and
 $t^2 \partial_t + nt u^a\partial_{u^a}$.
 
 The associated graded $\tgr(\fg) = \bop_{i \in \bbZ} \tgr_i(\fg)$,
 defined by $\tgr_i(\fg) := \fg^i / \fg^{i+1}$, is a graded Lie
 algebra with $\fm := \tgr_-(\fg)$ a NGLA.  The symbol algebra (\S
 \ref{S:geo-str}) of $(\cE,D)$ associated to {\em any} ODE \eqref{E:ODE}
 is isomorphic to $\fm$.  On $\tgr(\fg)$, the induced $\fp$-action has
 $\fg^1 \subset \fp$ acting trivially, so $\tgr_0(\fg) = \fg^0 /
 \fg^1$ acts on $\tgr(\fg)$ by grading-preserving derivations.
 
 It is convenient to introduce a grading directly on $\fg$, but since
 this is not $\fp$-invariant, it should only be regarded as an
 auxilliary structure.  Consider $\sfZ = -\frac{\sfH}{2} - ( 1 +
 \frac{n}{2} ) \id_m$.  The eigenvalues of $\ad_\sfZ$ introduce a Lie
 algebra grading $\fg = \fg_{-n-1} \op ... \op \fg_0 \op \fg_1$, so
 each $\fg_i$ is a $\fg_0$-module.  This satisfies $\fg^i = \bop_{j
   \geq i} \fg_j$ so that $\tgr_i(\fg) \cong \fg_i$.  As vector
 spaces,
 \begin{align}
 \begin{split}
 \fg_1 &\cong \bbR \sfY  \label{E:grg}\\
 \fg_0 &\cong \bbR \sfH \op \fgl_m\\
 \fg_{-1} &\cong \bbR \sfX \op (\bbR \sfv^n \otimes W) \\
 \fg_i &\cong \bbR \sfv^{n+1+i} \otimes W, \quad\quad i=-2, ..., -n-1.
 \end{split}
 \end{align}
 (We caution that $\sfX \in \fg_{-1}$ has usual $\fsl_2$-weight $+2$.)
 
 To pass to the group level, consider the natural action of $\tGL_2
 \times \tGL_m$ on $V_n \otimes W$ with kernel $T = \{ \lambda\,
 \id_2 \times \lambda^{-n}\, \id_m : \lambda \in \bbR^\times \}$, and
 $\LT_2 \subset \tGL_2$ (resp. $\LT^+_2$) the lower triangular
 (resp. strictly lower triangular) matrices.  Define
 \begin{align}
 \begin{split} \label{E:GP}
 G &= (\tGL_2 \times \tGL_m) / T \ltimes (V_n \otimes W),\\
 P &= (\LT_2 \times \tGL_m) / T,\\
 P_+ &= \LT^+_2 / T.
 \end{split}
 \end{align}
 Then $P_+ \subset P \subset G$ are closed subgroups in $G$
 corresponding to $\fg^1 \subset \fg^0 \subset \fg$, with $P_+$ normal
 in $P$. The adjoint action of $G$ restricts to a
 filtration-preserving $P$-action on $\fg$, and $P_+$ consists exactly
 of those elements for which the induced action on the associated graded
 $\tgr(\fg)$ is trivial. Thus we obtain a natural induced
 action of $G_0:=P/P_+$ on $\tgr(\fg)$. It is a familiar fact about
 parabolic subgroups that the quotient projection $P\to G_0$
 splits. Indeed, $G_0$ can be identified with the subgroup of those
 elements of $P$ whose adjoint action preserves the grading on $\fg$,
 and $(g,X)\mapsto g\exp(X)$ defines a diffeomorphism
 $G_0\times\fg_1\to P$. In this picture, $G_0 \subset P$ is the direct
 product of diagonal $2\times 2$ matrices and $\tGL_m$ (modulo
 $T$). This Lie group $G_0$ is isomorphic to that used in \S
 \ref{S:geo-str}, with $\bbR^\times$-factor there corresponding to
 elements $\diag(\lambda,1) \times \id_m$ (modulo $T$).
 The Lie algebra of $G_0$ is $\fg_0$. Collecting the results of this
 section, we in particular easily get:
 \begin{prop}\label{P:adm-pair}
   For the Lie algebra $\fg$ and the group $P$ defined above,
   $(\fg,P)$ is an admissible pair in the sense of Definition 2.5 of
   \cite{Cap2016}. Moreover, the group $P$ is of split exponential
   type in the sense of Definition 4.11 of that reference. 
 \end{prop}

 \subsection{Canonical Cartan connections}\label{S:Cartan}

 The equivalent description of (systems of) ODE as filtered
 $G_0$--structures that we have derived so far in particular includes
 a principal $G_0$--bundle $\cG_0\to\cE$. A particularly nice way to
 obtain invariants in such a situation is to construct a canonical
 Cartan geometry out of the filtered $G_0$--structure. In the language
 of \cite{Cap2016}, we are looking for a Cartan geometry of type
 $(\fg,P)$ (where $\fg$ and $P$ are as in \S \ref{S:triv-eq}
 above), which makes sense on smooth manifolds $M$ of dimension
 $\dim(\fg/\fp)$. Such a Cartan geometry then consists of a (right)
 principal $P$--bundle $\cG \to M$ and a \textit{Cartan connection}
 $\omega \in \Omega^1(\cG,\fg)$. This
 means that $\omega$ satisfies
 \begin{enumerate}
 \item For any $u \in \cG$, $\omega_u : T_u \cG \to \fg$ is a linear
   isomorphism;
 \item $\omega$ is $P$-equivariant, i.e.\ $R_g^* \omega = \Ad_{g^{-1}}
   \circ \omega$ for any $g \in P$;
 \item $\omega$ reproduces the generators of the fundamental vector fields
   $\zeta_A$, i.e.\ we have $\omega(\zeta_A) = A$ for any $A \in \fp$.
 \end{enumerate}
 The fundamental invariant available in this setting then is the
 curvature $K \in \Omega^2(\cG,\fg)$ of $\omega$, which is defined by
 $K(\xi,\eta) = d\omega(\xi,\eta) + [\omega(\xi),\omega(\eta)]$. The
 two--form $K$ is $P$-equivariant and horizontal, and can be
 equivalently encoded as the \textit{curvature function} $\kappa : \cG
 \to \bigwedge^2 \fg^* \otimes \fg$, defined by $\kappa(A,B) =
 K(\omega^{-1}(A), \omega^{-1}(B))$ for $A,B \in \fg$. The Cartan
 connection $\omega$ is {\em regular} if $\kappa(\fg^i, \fg^j) \subset
 \fg^{i+j+1}$ for all $i,j$.
 
 As detailed in Theorem 2.9 of \cite{Cap2016}, any regular Cartan
 geometry of type $(\fg,P)$ on a smooth manifold $M$ gives rise to an
 underlying filtered $G_0$--structure. The filtration $\{ T^i M \}$
 on $TM$ is obtained by projecting down the subbundles
 $T^i\cG:=\omega^{-1}(\fg^i)\subset T\cG$. Regularity of $\omega$
 implies that the symbol algebra $\tgr(TM)$ is everywhere
 NGLA-isomorphic to $\tgr_-(\fg)$. The reduction of structure group is
 then defined by the $G_0$-bundle $\cG_0 := \cG / P_+$. 

 Constructing a canonical Cartan connection means reversing this
 process. Given a filtered $G_0$--structure on $M$, one tries to
 extend the principal $G_0$--bundle $\cG_0\to M$ to a principal
 $P$--bundle $\cG\to M$, and endow that bundle with a natural Cartan
 connection. Such a construction was first obtained in \cite{DKM1999} for (systems of) ODE
 based on the general theory developed in \cite{Mor1993}. Here we
 follow the recent general construction in \cite{Cap2016}, which
 provides a more explicit characterization of the canonical Cartan
 connection via its curvature and stronger uniqueness results. 

 In view of Proposition \ref{P:adm-pair}, two more ingredients are
 needed to apply the general results of \cite{Cap2016}. On the one
 hand, we have to verify that the associated graded $\tgr(\fg)$ from
 \S \ref{S:triv-eq} is the full prolongation of its non--positive
 part (see Definition 2.10 of \cite{Cap2016}). On the other hand, we
 have to construct an appropriate normalization condition to be
 imposed on the curvature of the canonical Cartan connection. Both
 these steps are purely algebraic and we will carry them out in
 \S \ref{S:cd-norm} below. Using the results of Propositions
 \ref{P:codiff} and \ref{P:Tanaka} from there, we can apply Theorem
 4.12 of \cite{Cap2016} to obtain the following result.
 
 \begin{theorem} \label{T:G0str-Cartan} Fix $\fg$ and $P$ as in
   \S \ref{S:triv-eq}.  Then there is an equivalence of
   categories between filtered $G_0$--structures and regular, normal
   Cartan geometries of type $(\fg,P)$.
 \end{theorem}

 \begin{remark}\label{R:strong-reg}
   As described in \S \ref{S:geo-str}, ODE (considered up to
   contact transformations) define filtered $G_0$--structures, but not
   every filtered $G_0$--structure is of that form. This can be easily
   seen from the curvature of the canonical Cartan connection. We claim
   that for structures induced by ODE, we get a stronger version of
   regularity. Indeed, in this case $\kappa(\fg^i,\fg^j) \subset
   \fg^{\min(i,j)-1}$ for all $i,j<0$ and this is a proper subspace of
   $\fg^{ i+j+1}$ if $i,j<-1$.
   
   By definition of the curvature, we get
   \begin{align} \label{E:kappa}
   \kappa(\omega(\xi),\omega(\eta)) = \xi \cdot \omega(\eta) - \eta \cdot \omega(\xi) - \omega([\xi,\eta]) + [\omega(\xi),\omega(\eta)],
   \end{align}
   and if $\omega(\xi)$ has values in $\fg^i$ and $\omega(\eta)$ has values in $\fg^j$, then the first two summands on the right hand side have values in $\fg^{\min(i,j)-1}$.  Next, because of the large abelian
   ideal $\fa$, the Lie bracket on $\fg$ has the property that
   $[\fg^i,\fg^j]\subset\fg^{\min(i,j)-1}$ for all $i,j<0$, which handles the last term on the right hand side.  Hence, it remains to show that for structures coming from ODE, we also have $\omega([\xi,\eta])$ taking values in $\fg^{\min(i,j)-1}$.

   For such structures, we have the
   decomposition $T^i\cE=E\oplus F^i$ for all $i<0$ with $E\subset
   T^{-1}\cE$ and $F^i$ involutive. Given a vector field $\xi$ on
   $\cG$ such that $\omega(\xi)$ has values in $\fg^i$ for $i<0$, we
   can correspondingly decompose $\xi_1+\xi_2$, where $\xi_1$ is a
   lift of a section of $E\to\cE$ and $\xi_2$ lifts a section of
   $F^i\to\cE$. Similarly decompose $\eta=\eta_1+\eta_2$ for
   $\eta\in\fX(\cG)$ such that $\omega(\eta)$ has values in
   $\fg^j$. Using that a Lie bracket of lifts is a lift of  the Lie
   bracket of the underlying fields, one easily verifies that all the
   brackets $[\xi_i,\eta_j]$ are lifts of sections of
   $T^{\min(i,j)-1}\cE$ (or of smaller filtration components). Thus
   $\omega([\xi,\eta])$ has values in $\fg^{\min(i,j)-1}$, 
   which completes the argument. 

   All the further developments in this article make sense for
   arbitrary filtered $G_0$--structures and not only for the ones
   coming from ODE provided that one uses the description of Wilczynski
   invariants in Theorem \ref{T:Wilc} as a definition in the more
   general setting.
 \end{remark}
 
 \subsection{The space of solutions and C-class}
 \label{S:soln-space-geo-str}
 
 In the description of \S \ref{S:geo-str}, it is clear how to
 obtain the space of all solutions of \eqref{E:ODE}. The solutions are
 the integral curves of the line bundle $E\subset T\cE$ spanned by
 $\frac{d}{dt}$. Hence locally the space of solutions is the space of
 leaves of the foliation defined by $E$. In the case of the trivial
 equation $\bu^{(n+1)} = 0$, we obtain the solutions $\bu =
 \sum_{i=0}^n \ba_i t^i$, where $\ba_i \in W = \bbR^m$ are
 constant. Hence we obtain a global space $\cS$ of solutions in this
 case and viewing $\cE$ as $G/P$, we see that $\cS=G/Q$, where $Q =
 (\tGL_2 \times \tGL_m) / T \subset G$. This means that $\cS$ is the
 homogeneous model for Cartan geometries of type $(G,Q)$. In
 particular, the tangent bundle of $\cS$ is the homogeneous vector
 bundle $G\times_Q(\fg/\fq)$, and as a $Q$--module, we get
 $\fg/\fq\cong \fa = V_n\otimes W$. 

 This tensor decomposition of $\fg/\fq$ gives rise to a geometric structure on
 $\cS$ that can be described by the corresponding decomposition of
 the tangent bundle $T\cS$ into a tensor product. A simpler
 description is provided by the distinguished variety in $\bbP(\fa)$
 given as
 \begin{align} \label{E:Segre}
  \bbP^1\times \bbP^{m-1} \to \bbP(\fa), \quad 
  ([b_0:b_1],[w]) \mapsto [(b_0\sfx+b_1\sfy)^n \otimes w].
  \end{align} 
  Translating by $G$, one obtains a canonical isomorphic copy of this
  variety in each tangent space of $\cS$. The resulting geometric
  structure is called a {\em Segr\'e structure} (modelled on
  \eqref{E:Segre}). When $m=1$, these structures are commonly called
  {\em $\tGL_2$-structures}, but we will use the term Segr\'e
  structure for all cases. Notice that this is a standard first order
  structure corresponding to $Q\subset \tGL(\fa)$, without any additional
  filtration on the tangent bundle. 

  Now one may ask the question whether similar things happen for more general ODE, both on the level of Cartan geometries and on
  the level of Segr\'e structures.  On the latter level, this is studied intensively in the literature in many special cases, see e.g. \cite{DT2006,Nur2009,GN2010,Kry2010,DG2012}.  For our purposes, the results of \cite{Dou2008} are particularly relevant. In that article, it is shown in general that vanishing of the generalized Wilczynski invariants from Definition \ref{D:Wilc}
  implies existence of a natural Segr\'e structure on the space of
  solutions. The pullback of $T\cS$ to $\cE$ is naturally isomorphic to 
  $T\cE/E$. This is modelled on $\fa$, so on that level a
  decomposition as a tensor product is available. The Wilczynski
  invariants can be interpreted as obstructions to this decomposition
  descending to a decomposition of $T\cS$, which is crucial for the
  developments in \cite{Dou2008}, compare also to the proof of Theorem \ref{T:Wilc}.

  On the level of Cartan geometries the question of descending is
  closely related to the concept of C-class. The technical aspects of
  this descending process are worked out in the case of parabolic
  geometries in \cite{Cap2005}. As shown in \S 1.5.13 and 1.5.14
  of \cite{CS2009}, the proofs in that article apply to general
  groups. Consider an equation $\cE$ and a (local) space of solutions
  $\cS$, i.e.~a local leaf space for $E\subset T\cE$. Descending of
  the Cartan geometry $(\cG\to\cE,\omega)$ first requires that the
  principal right action of $P$ on $\cG$ extends to a smooth action of
  $Q\supset P$ which has the fields $\omega^{-1}(A)\in\fX(\cG)$ for $A \in \fq$ as
  fundamental vector fields. If such an extension exists, then,
  possibly shrinking $\cS$, one obtains a projection $\cG\to\cS$,
  which is a $Q$--principal bundle. Next, one has to ask whether (the
  restriction of) $\omega$ can be interpreted as a Cartan connection
  on that principal $Q$--bundle, which boils down to the question of
  $Q$--equivariance. Surprisingly, it turns out that the whole question
  of descending of the Cartan geometry is equivalent to the fact that
  all values of the curvature function $\kappa$ of $\omega$ vanish
  upon insertion of any element of $\fq/\fp\subset\fg/\fp$, see
  Theorem 1.5.14 of \cite{CS2009}. 

  But now the fact that the canonical Cartan geometry on $\cE$
  descends to the space $\cS$ implies that the Cartan curvature and
  hence all invariants derived from it in an equivariant fashion
  descend to $\cS$ and thus are first integrals. This is the technical
  definition of C-class that we use in this article: 

  \begin{defn} \label{D:C-class} An ODE \eqref{E:ODE} is of {\em
      C-class} if its corresponding regular, normal Cartan geometry
    $(\cG \to \cE, \omega)$ descends to sufficiently small spaces of
    solutions or, equivalently, if its curvature function satisfies
    $i_\sfX \kappa = 0$, where $\sfX \in \fg_{-1}$ was defined in \S \ref{S:triv-eq}.
 \end{defn}
 
 \section{Codifferentials and normalization conditions}
 \label{S:cd-norm}

 \subsection{Filtrations and gradings}\label{S:filtgr}
 We will use the general results from \cite{Cap2016} to obtain
 canonical Cartan connections. In addition to the properties of the
 pair $(\fg,P)$ that we have already verified, the main ingredient needed
 to apply this method is a choice of normalization condition. We do
 this via a codifferential in the sense of Definition 3.9 of
 \cite{Cap2016}.

 Such a codifferential consists of $P$--equivariant maps acting between
 spaces of the form $L(\bigwedge^k(\fg/\fp),\fg)$ of alternating multilinear
 maps. An important role in \cite{Cap2016} is played by the natural
 $P$--invariant filtration on these spaces and the associated graded
 spaces. For our purposes, it will be useful to view these as subspaces
 of the chain spaces $C^k(\fg,\fg)=L(\bigwedge^k\fg,\fg)$. Hence we
 will first collect the necessary information on filtrations and
 associated graded spaces in this setting. Observe that each of the
 spaces $C^k(\fg,\fg)$ naturally is a representation of $\fg$ and of
 $P$, and we can identify $L(\bigwedge^k(\fg/\fp),\fg)$ with the
 subspace
$$
C^k_{\text{hor}}(\fg,\fg)=\{ \phi \in C^k(\fg,\fg) : i_z \phi =
 0, \forall z \in \fp \}
$$
of horizontal $k$--chains, which is immediately seen to be
$P$--invariant.

As we have seen in \S \ref{S:triv-eq}, the Lie algebra $\fg$
carries a $P$--invariant filtration $\{\fg^i\}_{i=-n-1}^1$ such that
$\fp = \fg^0$. Moreover, we noticed that this filtration is actually
induced by a grading $\fg=\fg_{-n-1}\oplus\dots\oplus\fg_1$ of $\fg$ in
the sense that $\fg^i=\oplus_{j\geq i}\fg_j$. The grading is not
$P$--invariant, however, so it has to be viewed as an auxilliary
object. In particular, this implies that one can identify the filtered
Lie algebra $\fg$ with its associated graded Lie algebra
$\tgr(\fg)$. The filtration and the grading on $\fg$ induce a
filtration and a grading on each of the chain spaces $C^k(\fg,\fg)$,
which can be conveniently described in terms of homogeneity. Moreover,
it follows readily that each of the spaces $C^k(\fg,\fg)$ can be
naturally identified with its associated graded.

The notion of homogeneity is more familiar in the setting of
gradings: We say that $\varphi \in C^k(\fg,\fg)$ is {\em homogeneous of
degree $\ell$} if, for all $i_1,\dots,i_k\in\{-n-1,\dots,1\}$, it maps
$\fg_{i_1}\times\dots\times\fg_{i_k}$ to $\fg_{i_1+\dots+i_k+\ell}$. In our
simple situation, {\em homogeneity of degree $\geq\ell$ (in the filtration
sense)} then simply means that $\fg_{i_1}\times\dots\times \fg_{i_k}$ is always
mapped to $\fg^{i_1+\dots+i_k+\ell}$. For the passage to the
associated graded, it suffices to consider spaces of the form
$L(\bigwedge^k(\fg/\fp),\fg)$. As proved in Lemma 3.1 of
\cite{Cap2016}, identifying $\fg$ with $\tgr(\fg)$, the associated
graded to this filtered space can be identified with $C^k(\fg_-,\fg)$
(with its natural grading). For a map $\varphi\in
L(\bigwedge^k(\fg/\fp),\fg)$ which is homogeneous of degree
$\geq\ell$, the projection $\tgr_\ell(\varphi)\in C^k(\fg_-,\fg)_\ell$
is obtained by applying $\varphi$ to (the classes of) elements of
$\fg_-$ and taking the homogeneous component of degree $\ell$. Here we
denote by $C^k(\fg_-,\fg)_\ell$ the homogeneity $\ell$ component of
$C^k(\fg_-,\fg)$. 

The spaces $C^k(\fg_-,\fg)$ are the chain spaces in the standard
complex computing the Lie algebra cohomology $H^*(\fg_-, \fg)$ of
the Lie algebra $\fg_-$ with coefficients in the module
$\fg$. Correspondingly, there is a standard differential in this
complex, which we denote by $\partial_{\fg_-}$. This differential
plays an important role in the definitions of normalization conditions
and of codifferentials.

\subsection{Scalar product and codifferential}\label{S:codiff}
 As in \S \ref{S:filtgr}, we identify $L(\bigwedge^k(\fg/\fp),\fg)$ with $C^k_{\text{hor}}(\fg,\fg)$.  Define an inner product $\langle\ , \ \rangle$ on $\fg$ by declaring
 $\sfX$, $\sfH$, $\sfY$, $\sfv^i_b:=\sfv^i\otimes\sfe_b$, $\sfe^a_b$
 to be an orthogonal basis with
 \[
 \langle \sfX, \sfX \rangle = \langle \sfY, \sfY \rangle = 1, \quad
 \langle \sfH, \sfH \rangle = 2, \quad
 \langle \sfe^a_b, \sfe^a_b \rangle = 1, \quad 
 \langle \sfv^i_b, \sfv^i_b \rangle = \frac{(n-i)!}{i!}.
 \]
 Then $\forall A,B \in \fq$ and $\forall u,v \in \fa$, this satisfies:
 \begin{align} \label{E:innprod}
 \langle A, B \rangle = \tr(A^\top B), \qquad
 \langle Au, v \rangle = \langle u, A^\top v \rangle.
 \end{align}
 Extend $\langle \ , \ \rangle$ to an inner product on $C^*(\fg,\fg)$.
 The spaces $C^k(\fg,\fg)$ are the chain spaces in the standard
 complex computing the Lie algebra cohomology $H^*(\fg,\fg)$, and we
 denote by $\partial_{\fg}$ the standard differentials in that
 complex. From the explicit formula for these differentials (which
 only uses the Lie bracket in $\fg$), it follows readily that these
 maps are $\fg$--equivariant and $Q$--equivariant.

 \begin{defn}\label{D:codiff}
   For each $k$, we define the \textit{codifferential}
   $\partial^*:C^k(\fg,\fg)\to C^{k-1}(\fg,\fg)$ as the adjoint (with
   respect to the inner products we have just defined) of the Lie
   algebra cohomology differential $\partial_{\fg}$. Explicitly, we
   have the relation $\langle \partial_{\fg} \phi, \psi \rangle = \langle
   \phi, \partial^* \psi \rangle$ for all $\phi\in C^{k-1}(\fg,\fg)$
   and $\psi\in C^k(\fg,\fg)$.
 \end{defn}
 
 \begin{lemma}\label{L:codiff} 
   The codifferential restricts to a $P$--equivariant map
   $\partial^*:L(\bigwedge^k(\fg/\fp),\fg)\to
   L(\bigwedge^{k-1}(\fg/\fp),\fg)$. This map preserves homogeneity
   and thus is compatible with the filtrations on both
   spaces. Moreover, it is image--homogeneous in the sense of
   Definition 3.7 of \cite{Cap2016}.
 \end{lemma}

 \begin{proof} We have already noted that $\partial_{\fg}$ is
   $\fg$-equivariant. Now for any $A \in \fq$, we have
 \begin{align*}
  \langle \phi, A \partial^* \psi \rangle &= \langle A^\top \phi,
  \partial^* \psi \rangle = \langle \partial_{\fg}(A^\top \phi), \psi
  \rangle = \langle A^\top \partial_{\fg} \phi, \psi \rangle\\ &= \langle
  \partial_{\fg} \phi, A \psi \rangle = \langle \phi, \partial^*(A \psi)
  \rangle,
 \end{align*}
 so $\partial^*$ is $\fq$-equivariant on the full cochain spaces. Since
 the grading element $\sfZ$ lies in $\fz(\fg_0) \subset \fq$, we see
 that $\partial^*$ commutes with the action of $\sfZ$. This means that
 it preserves homogeneity in the graded--sense and thus also in the
 sense of filtrations.
 
 Let $\op^\perp$ denote orthogonal direct sum.  Then $\fg = \fg_-
 \op^\perp \fp$ induces $\fg^* = \fann(\fg_-) \op^\perp \fann(\fp)$.
 Letting $\bigwedge^{i,j} := \bigwedge^i \fann(\fg_-) \otimes
 \bigwedge^j \fann(\fp)$, we have $\bigwedge^k \fg^* \cong
 \bop_{i+j=k}^\perp \bigwedge^{i,j}$. Now by definition, the subspace
 $L(\bigwedge^k(\fg/\fp),\fg)$ of $C^k(\fg,\fg)$ coincides with
 $\bigwedge^{0,k} \otimes \fg$. Thus, its orthocomplement is given by
 $\bop_{i>0}^\perp \bigwedge^{i,k-i}\otimes\fg$, and this space can be written
 as
$$
\{ \varphi \in C^k(\fg,\fg) : \varphi(v_1,...,v_k) = 0, \,
 \forall v_i \in \fg_- \}.
$$ 
Since $\fg_-$ is a subalgebra of $\fg$, the definition of the
differential implies that $\partial_{\fg}$ maps
$L(\bigwedge^{k-1}(\fg/\fp),\fg)^\perp$ to
$L(\bigwedge^k(\fg/\fp),\fg)^\perp$. Now for $\psi \in
L(\bigwedge^k(\fg/\fp),\fg)$, we can verify that $\partial^*\psi\in
L(\bigwedge^{k-1}(\fg/\fp),\fg)$ by showing that for all $\phi\in
L(\bigwedge^{k-1}(\fg/\fp),\fg)^\perp$, we get
$0=\langle\phi,\partial^*\psi\rangle$. But this follows directly from
the definition as an adjoint. Since $L(\bigwedge^k(\fg/\fp),\fg)$ is a
$P$--invariant subspace of $C^k(\fg,\fg)$ for each $k$,
$Q$--equivariance of $\partial^*$ on $C^k(\fg,\fg)$ readily implies
$P$--equivariance of the restriction.

Image-homogeneity as defined in \cite{Cap2016} requires the
following. If we have an element in the image of $\partial^*$, which
is homogeneous of degree $\geq\ell$ in the filtration sense, then it
should be possible to write it as the image under $\partial^*$ of an
element which itself is homogeneous of degree $\geq\ell$. But in our
case, the filtration is derived from a grading that is preserved by
$\partial^*$ . Thus, if all non--zero homogeneous components of
$\partial^*\psi$ lie in degrees $\geq\ell$, it follows that all
homogeneous components of degree $<\ell$ of $\psi$ must be contained
in the kernel of $\partial^*$. (Otherwise, their images would be of
the same homogeneity.) Hence the homogeneous components of degree
$<\ell$ can be left out without changing the image, and image-homogeneity follows. 
\end{proof}

To prove that $\partial^*$ can be used to obtain a normalization
condition, we have to consider the induced maps between the associated
graded spaces. As in \S \ref{S:filtgr}, we
  view the associated graded of $L(\bigwedge^k(\fg/\fp),\fg)$ as
  $C^k(\fg_-,\fg)$.  Observe further that $C^k(\fg_-,\fg)$ is
exactly the subspace $\bigwedge^{0,k}\otimes\fg\subset C^k(\fg,\fg)$
as introduced in the proof of Lemma \ref{L:codiff}. Having made these
observations we can now verify the remaining properties of the
codifferential needed in order to apply the general theory for
existence of canonical Cartan connections.

\begin{prop}\label{P:codiff}
  The maps $\partial^*$ from Definition \ref{D:codiff}
  define a codifferential in the sense of Definition 3.9 of
  \cite{Cap2016}. Hence, in the terminology of that reference,
  $\ker(\partial^*)\subset L(\bigwedge^2(\fg/\fp),\fg)$ is a
  normalization condition and $\im(\partial^*)\subset\ker(\partial^*)$
  is a maximally negligible submodule.
\end{prop}
\begin{proof}
  In view of Lemma \ref{L:codiff}, it remains to verify the second
  condition in Definition 3.9 of \cite{Cap2016}. This says that the
  maps $\underline{\partial^*}:C^k(\fg_-,\fg)\to C^{k-1}(\fg_-,\fg)$
  induced by $\partial^*$ are disjoint to $\partial_{\fg_-}$. As
  above, we can identify $C^k(\fg_-,\fg)$ with the subspace
  $\bigwedge^{0,k}\otimes\fg\subset C^k(\fg,\fg)$, which endows it
  with an inner product. Since $\fg_-$ is a subalgebra in $\fg$, it
  easily follows from the definition of the Lie algebra cohomology
  differential that for $\psi\in C^k(\fg_-,\fg)\subset C^k(\fg,\fg)$
  we get $\partial_{\fg}\psi\in \bigwedge^{1,k}\otimes\fg\oplus
  \bigwedge^{0,k+1}\otimes\fg$. Moreover, the component of
  $\partial_{\fg}\psi$ in $\bigwedge^{0,k+1}\otimes\fg$ coincides with
  $\partial_{\fg_-}\psi$. 

  Now taking $\varphi\in C^k(\fg_-,\fg)$ that is homogeneous of some
  fixed degree $\ell$, we get $\underline{\partial^*}\varphi$ by
  interpreting $\partial^*\varphi$ as an element of
  $C^{k-1}(\fg_-,\fg)$. For $\psi\in C^{k-1}(\fg,\fg)$, we thus can
  have $\langle\partial^*\varphi,\psi\rangle\neq 0$ only if $\psi$ is
  homogeneous of the same degree $\ell$ and contained in
  $\bigwedge^{0,k-1}\otimes\fg$. By definition, we get
  $\langle\partial^*\varphi,\psi\rangle=\langle\varphi,
  \partial_{\fg}\psi\rangle$. Since
  $\varphi\in\bigwedge^{0,k}\otimes\fg$, we may replace
  $\partial_{\fg}\psi$ by its component in that subspace and hence by
  $\partial_{\fg_-}\psi$. This shows that $\underline{\partial^*}$ is
  adjoint to $\partial_{\fg_-}$, which implies the required
  disjointness. All remaining claims now follow directly from
  Proposition 3.10 of \cite{Cap2016}.
\end{proof}

 As noted in \S \ref{S:Cartan}, the curvature
 of a Cartan geometry is encoded in the curvature function $\kappa$,
 which has values in $L(\bigwedge^2(\fg/\fp),\fg)$. Normality of the
 Cartan geometry then exactly means that the values of $\kappa$
 actually lie in the subspace $\ker(\partial^*)$.

\subsection{Lie algebra cohomology and Tanaka
  prolongation}\label{S:Tanaka}
To obtain a more explicit description of the codifferential
$\partial^*$, we next study the Lie algebra cohomology differential
$\partial_{\fg_-}$. This will also allow us to verify that
$\tgr(\fg)\cong\fg$ is the full prolongation of its non--positive
part, which is the last ingredient needed to prove Theorem
\ref{T:G0str-Cartan}. This can be expressed in terms of the Lie
algebra cohomology $H^*(\fg_-,\fg)$. 

Recall that $\fg=\fq\ltimes\fa$, with the abelian ideal
$\fa=V_n\otimes\bbR^m$ and the reductive subalgebra
$\fq=\mathfrak{sl}_2\times\mathfrak{gl}_m$. Moreover, $\fp\subset\fq$
and $\fg_-=\bbR\cdot\sfX\oplus\fa$. Now proceeding similarly as above,
we view $L(\bigwedge^k(\fg/\fq),\fg)$ as the subspace of
$C^k(\fg,\fg)$ consisting of those maps which vanish upon insertion of
one element of $\fq$. This can then be identified with the chain
space $C^k(\fa,\fg)$, where we view $\fg$ as an $\fa$--module via the
adjoint action. This identification is even $\fq$--equivariant, since
$\fg=\fq\oplus\fa$ as a $\fq$--module. 

On the chain spaces $C^*(\fa,\fg)$, we again have a Lie algebra
cohomology differential, which we denote by
$\partial_{\fa}$. Explicitly, this differential is given by
\[
\partial_{\fa}\varphi(X_0,\dots,X_k)=\textstyle
\sum_i(-1)^i[X_i,\varphi(X_0,\dots,\widehat{X_i},\dots,X_k)]. 
\]
Now define $\omega^\sfX\in{\fg_-}^*$ to be the functional sending
$\sfX$ to $1$ and vanishing on $\fa\subset\fg_-$.  Given $\phi \in
C^k(\fg_-,\fg)$, we have $\phi = \omega^\sfX \wedge \phi_1 + \phi_2$,
for elements $\phi_1 \in C^{k-1}(\fa,\fg)$ and $\phi_2 \in
C^k(\fa,\fg)$. Explicitly, we have $\phi_1=i_{\sfX}\phi$ and
$\phi_2=\phi-\omega^\sfX \wedge \phi_1$. We express this by writing
$\phi = \begin{psmallmatrix} \phi_1\\ \phi_2 \end{psmallmatrix}$.

\begin{lemma}\label{L:delg-}
  In terms of the notation just introduced, the Lie algebra cohomology
  differential $\partial_{\fg_-}$ is given by
 \begin{equation} \label{E:delg-}
   \partial_{\fg_-}\begin{pmatrix} \phi_1 \\ \phi_2\end{pmatrix}
   = \begin{pmatrix} -\partial_\fa \phi_1 +
     \sfX\cdot\phi_2\\ \partial_\fa \phi_2\end{pmatrix}. 
 \end{equation}
\end{lemma}
\begin{proof}
  Take $\phi_1,\phi_2 \in C^*(\fa,\fg)$ of degrees $k-1$ and $k$,
  respectively. Evaluating on $\bigwedge^{k+1} \fa$, we clearly have
  $\partial_{\fg_-}(\omega^\sfX \wedge \phi_1) = 0$ and
  $\partial_{\fg_-} \phi_2 = \partial_\fa \phi_2$. Next, simple direct
  computations show that for elements $v_i \in \fa$, we obtain
\[
\partial_{\fg_-}(\omega^\sfX \wedge \phi_1)(\sfX,v_1,...,v_k) =
-\partial_\fa\phi_1(v_1,\dots,v_k),
\]
while $(\partial_{\fg_-} \phi_2)(\sfX,v_1,...,v_k)$ equals
 \begin{align*}
   \sfX \cdot (\phi_2&(v_1,...,v_k)) + \textstyle\sum_{j=1}^k (-1)^j \phi_2([\sfX,v_j],v_1,...,\widehat{v}_j,...,v_k)\\
   =&(\sfX\cdot \phi_2)(v_1,...,v_k).
\end{align*}
\end{proof}

The $\fq$--equivariant decomposition $\fg=\fq\oplus\fa$ also induces a
decomposition of $C^k(\fa,\fg)$ according to the values of multilinear
maps.  While the first factor is not a space of cochains, we still
denote this decomposition by $C^k(\fa,\fg)=C^k(\fa,\fq)\oplus
C^k(\fa,\fa)$. Observe that from the definition of $\partial_{\fa}$ it
follows readily that $C^k(\fa,\fa)\subset\ker(\partial_{\fa})$ and
that $\im(\partial_{\fa})\subset C^{k+1}(\fa,\fa)$. Using this, we can
now formulate the result on the full prolongation.
  
\begin{prop}[Tanaka prolongation]\label{P:Tanaka}  Let $n \geq 2$, $m \geq 1$, with $(n,m) \neq (2,1)$.  Then the graded Lie algebra
  $\tgr(\fg)\cong\fg$ is the full prolongation of its non--positive
  part.
 \end{prop}
 \begin{proof} 
   It is well known that the statement is equivalent to the fact that
   $H^1(\fg_-,\fg)$ is concentrated in non--positive homogeneities,
   compare with Proposition 2.12 of \cite{Cap2016}. In the vector
   notation introduced above, an element of $C^1(\fg_-,\fg)$ can be
   written as $\binom{A}{\phi}$ for $A\in\fg=C^0(\fa,\fg)$ and
   $\phi\in C^1(\fa,\fg)$. Now as indicated above, we can decompose
   $\phi=\phi_{\fa}+\phi_{\fq}$ according to the values. By Lemma
   \ref{L:delg-}, $0=\partial_{\fg_-}\binom{A}{\phi} = \begin{psmallmatrix} -\partial_\fa(A) + \sfX \cdot \phi\\ \partial_\fa \phi \end{psmallmatrix}$ implies
   $0=\partial_{\fa}\phi=\partial_{\fa}\phi_{\fq}$. But now by
   definition, the restriction $C^1(\fa,\fq)\to C^2(\fa,\fa)$ of
   $\partial_{\fa}$ is exactly the Spencer differential associated to
   $\fq\subset \fa^*\otimes\fa$.

   We note that $\fq$ acts irreducibly on $\fa$, and is not in the list of infinite-type algebras in \cite{KN1965}.  Given the assumptions on $m$ and $n$, $\fa^* \op \fq \op \fa$ is not a $|1|$-graded semisimple Lie algebra.  (The list of these algebras is well-known -- see\ \S 3.2.3 in \cite{CS2009}.) Thus, by the main result of  \cite{KN1964} by Kobayashi and Nagano, $\fq\subset \fa^*\otimes\fa$ has trivial first
   prolongation, so this Spencer differential is injective. 
   Hence, we conclude that $\phi_\fq = 0$.  
   
   We have already seen that
   $\sfX\cdot\phi=\partial_\fa(A)$. Now we can decompose the
   representation $\fa^*\otimes\fa$ of $\fq$ into irreducible
   components. Writing this as $\fq \op \oplus_j U_j$, we can accordingly decompose $\phi=B + \sum_j\phi_j$ and
   this decomposition is preserved by the action of $\sfX\in\fq$. But
   on the other hand, $\partial_\fa:\fg\to\fa^*\otimes\fa$ vanishes on
   $\fa\subset\fg$ and coincides with the inclusion on
   $\fq\subset\fg$. Thus we conclude that $\sfX\cdot\phi_j=0$ for all
   $j$, which means that these $\phi_j$ actually have to be
   contained in highest weight spaces for the action of $\fsl_2$. These
   are all represented by positive powers of $\sfX$ and thus contained in
   negative homogeneity.

   The upshot of this discussion is that if $\binom{A}{\phi}$ lies in
   the kernel of $\partial_{\fg_-}$ and has positive homogeneity, then
   $\phi=\phi_\fa \in C^1(\fa,\fa)$ must satisfy $\phi = \lambda\, \ad_\sfY|_\fa$, and hence
   $\sfX \cdot \phi = \lambda\, \ad_\sfH|_\fa$.  By the homogeneity assumption $A$
   has to be homogeneous of non-negative degree, hence lies in $\fq$, so
   $\partial_\fa(A) = \sfX \cdot \phi$ implies $A=-\lambda\sfH$.  But then
   one immediately verifies that $\binom{A}{\phi}=\partial_{\fg_-}
   (-\lambda\sfY)$, which completes the proof. 
 \end{proof}

 As we have observed in \S \ref{S:Cartan} already, this completes
 the proof of Theorem \ref{T:G0str-Cartan}, so we have an equivalence
 of categories between filtered $G_0$--structures and regular normal
 Cartan geometries.
 
\subsection{A codifferential formula}\label{S:codiff-explicit}
To proceed towards a more explicit description of the codifferential
$\partial^*$, we continue identifying $C^k(\fa,\fg)$ with the subspace
of $C^k(\fg,\fg)$ of those cochains which vanish under insertion of an
element of $\fq$. Doing this, we can restrict the inner product from
\S \ref{S:codiff} to the subspace $C^k(\fa,\fg)$ and define a map
$\partial^*_{\fa}$ as the adjoint of the Lie algebra cohomology
differential $\partial_{\fa}$. We further observe that the
decomposition $C^k(\fa,\fg)=C^k(\fa,\fq)\oplus C^k(\fa,\fa)$ is
orthogonal with respect to our inner product. The basic properties of
$\partial^*_{\fa}$ are as follows.

\begin{lemma}\label{L:q-valued} \quad

 \begin{enumerate}
 \item The map $\partial^*_{\fa}$ is $\fq$--equivariant. 
 \item For each $k$, we have $\im(\partial^*_{\fa})\subset
  C^k(\fa,\fq)\subset\ker(\partial^*_{\fa})$. 
 \item For $k=1$, we get $\im(\partial^*_{\fa})=C^1(\fa,\fq)$ and
  $\ker(\partial^*_{\fa})$ is the direct sum of $C^1(\fa,\fq)$ and the
  orthocomplement of $\fq\subset\fa^*\otimes\fa=C^1(\fa,\fa)$
  included via the natural action of $\fq$ on $\fa$.
 \end{enumerate}
\end{lemma}
\begin{proof}
  (1) is proved in exactly the same way as equivariance of the
  codifferential in Lemma \ref{L:codiff}.

  (2) In \S \ref{S:Tanaka} we have observed that
  $C^k(\fa,\fa)\subset\ker(\partial_{\fa})$ and
  $\im(\partial_{\fa})\subset C^k(\fa,\fa)$ for each $k$. By the
  definition as an adjoint, we see that
  $\ker(\partial^*_{\fa})=\im(\partial_{\fa})^\perp$ and
  $\im(\partial^*_{\fa})=\ker(\partial_{\fa})^\perp$. Thus (2) follows
  from the fact that $C^k(\fa,\fq)=C^k(\fa,\fa)^\perp$ for each $k$.

  (3) We have already observed in the proof of Proposition
  \ref{P:Tanaka} that $\partial_{\fa}:\fg\to C^1(\fa,\fg)$ vanishes on
  $\fa$ and restricts to the representation $\fq\to \fa^*\otimes\fa$ on
  $\fq$. Thus $\im(\partial_{\fa})=\fq\subset C^1(\fa,\fa)\subset
  C^1(\fa,\fg)$, which together with the arguments from (2) implies
  the claimed description of $\ker(\partial^*_{\fa})$.

  On the other hand, $\partial_{\fa}:C^1(\fa,\fg)\to C^2(\fa,\fg)$,
  vanishes on $C^1(\fa,\fa)$ while in the proof of Proposition
  \ref{P:Tanaka} we have seen that it restricts to an injection on
  $C^1(\fa,\fq)$. Thus $\ker(\partial_{\fa})=C^1(\fa,\fa)$, and the
  description of $\im(\partial^*_{\fa})$ in degree one follows.
\end{proof}
 
As above, we view $C^k(\fa,\fg)$ as the subspace of
$L(\bigwedge^k(\fg/\fp),\fg)$ consisting of those elements which
vanish under insertion of the element $\sfX$.  Given the basis $\{
\sfY,\sfH, \sfe^a_b, \sfX, \sfv^i_b \}$ of $\fg$, let $\{ \eta^\sfY,
\eta^\sfH, \eta^b_a, \omega^\sfX, \omega_i^b \}$ be the dual basis.  

\begin{prop} \label{P:hds} In terms of the notation from \S \ref{S:Tanaka}, 
  the codifferential $\partial^*$ (on horizontal $k$-forms) is given by
 \begin{equation} \label{E:hds}
   \partial^*\begin{pmatrix} \phi_1 \\ \phi_2\end{pmatrix} = \begin{pmatrix} -\partial^*_\fa \phi_1\\ \partial^*_\fa \phi_2 + \sfY \cdot \phi_1\end{pmatrix},
 \end{equation}
 where $\phi_1 \in C^{k-1}(\fa,\fg)$ and $\phi_2 \in C^k(\fa,\fg)$.
 \end{prop}
 \begin{proof} 
 Take $\psi_1 \in C^{k-2}(\fa,\fg)$ and $\psi_2 \in C^{k-1}(\fa,\fg)$ and put
   $\psi=\omega^\sfX\wedge\psi_1+\psi_2 \in C^{k-1}(\fg_-,\fg)$.  In the proof of Proposition
   \ref{P:codiff}, we have seen that $\partial_{\fg}\psi$ and
   $\partial_{\fg_-}\psi$ differ only by elements of
   $\bigwedge^{1,k-1}\otimes\fg$. Using Lemma \ref{L:delg-} we thus
   conclude that, up to terms involving elements of
   $\{\eta^\sfY,\eta^\sfH, \eta^b_a \}$, we get
 \[
 \partial_{\fg}(\omega^\sfX \wedge \psi_1 + \psi_2) \equiv\omega^\sfX
 \wedge ( -\partial_\fa \psi_1 + \sfX \cdot \psi_2) + \partial_\fa
 \psi_2.
 \]
 
 Since $\{ \eta^\sfY, \eta^\sfH, \eta^b_a \}$ is orthogonal to the
 horizontal forms $\{ \omega^\sfX, \omega_i^b \}$, the formula for
 $\partial^*$ (on horizontal forms) follows from:
  \begin{align*}
    \langle \partial^*(\omega^\sfX \wedge \phi_1), \omega^\sfX \wedge
    \psi_1 \rangle &= \langle \omega^\sfX \wedge
    \phi_1, \partial_{\fg}(\omega^\sfX \wedge \psi_1) \rangle
    = \langle \omega^\sfX \wedge \phi_1, -\omega^\sfX \wedge \partial_\fa \psi_1 \rangle\\
    & = -\langle \omega^\sfX, \omega^\sfX \rangle \langle
    \phi_1, \partial_\fa \psi_1 \rangle
    = -\langle \omega^\sfX, \omega^\sfX \rangle \langle \partial^*_\fa \phi_1, \psi_1 \rangle\\
    &=\langle -\omega^\sfX \wedge \partial^*_\fa \phi_1, \omega^\sfX \wedge \psi_1 \rangle\\
 \langle \partial^*(\omega^\sfX \wedge \phi_1), \psi_2 \rangle
 &= \langle \omega^\sfX \wedge \phi_1, \partial_\fg\psi_2 \rangle
 = \langle \omega^\sfX \wedge \phi_1, \omega^\sfX \wedge \sfX\cdot\psi_2 \rangle\\
 &= \langle \phi_1, \sfX\cdot\psi_2 \rangle = \langle \sfY \cdot \phi_1, \psi_2 \rangle\\
 \langle \partial^* \phi_2, \omega^\sfX \wedge \psi_1 \rangle &= \langle \phi_2, \partial_\fg(\omega^\sfX \wedge \psi_1) \rangle = 0\\
 \langle \partial^* \phi_2, \psi_2 \rangle &= \langle
 \phi_2, \partial_\fg\psi_2 \rangle = \langle
 \phi_2, \partial_\fa\psi_2 \rangle = \langle \partial^*_\fa \phi_2,
 \psi_2 \rangle
 \end{align*}
 \end{proof}
 
 \begin{cor}\label{C:red} 
   Consider $\im(\partial^*)\subset\ker(\partial^*)\subset
   L(\bigwedge^k(\fg/\fp),\fg)$. Then the natural representation of
   $P$ on $\ker(\partial^*) / \im(\partial^*)$ is completely
   reducible, i.e.\ $\fg^1$ acts trivially.
 \end{cor}
 
 \begin{proof}
   Let $\phi \in \ker(\partial^*)$.  From \eqref{E:hds},
   $\partial^*_\fa \phi_1 = 0$ and $\partial^*_\fa \phi_2 + \sfY \cdot
   \phi_1 = 0$.  Since $\sfY \cdot \omega^\sfX = -\omega^\sfX \circ
   \ad_\sfY = 0$, then $\sfY \cdot \phi = \omega^\sfX \wedge (\sfY
   \cdot \phi_1) + \sfY\cdot \phi_2$, so
 \[
 \sfY \cdot \phi = \begin{pmatrix} \sfY \cdot \phi_1 \\ \sfY \cdot
   \phi_2 \end{pmatrix} = \begin{pmatrix} -\partial_\fa^* \phi_2 \\
   \sfY\cdot \phi_2\end{pmatrix} = \partial^*\begin{pmatrix} \phi_2 \\
   0\end{pmatrix} \in \im(\partial^*).
 \]
 Hence, $\fg^1$ acts trivially on $\ker(\partial^*) /
 \im(\partial^*)$.
 \end{proof} 
 
  \subsection{Homogeneous examples of C-class ODE}
  \label{S:hom-C-class}

 It is well-known that the submaximal (contact) symmetry dimension for
 scalar ODE of order $\geq 4$ is two less than that of the (maximally
 symmetric) trivial equation, except for orders 5 and 7 where it is
 only one less \cite{Olv1995}.  For these cases, explicit submaximally symmetric models are well-known:
 \begin{align} 
 9 (u'')^2 u^{(5)} - 45 u'' u''' u'''' + 40 (u''')^3 = 0; \label{E:5-ex}\\
 10 (u''')^3 u^{(7)} - 70 (u''')^2 u^{(4)} u^{(6)} - 49 (u''')^2 (u^{(5)})^2 \label{E:7-ex}\\
 + 280 u''' (u^{(4)})^2 u^{(5)} - 175 (u^{(4)})^4 = 0. \nonumber
 \end{align}
 These have $A_2 \cong \fsl_3$ and $C_2 \cong \fsp_4$ symmetry respectively.  

 Doubrov \cite{Dou2008} showed that \eqref{E:5-ex} and \eqref{E:7-ex} are Wilczynski-flat.  We will describe their Cartan curvatures, observe the vanishing under $\sfX$-insertions, and hence confirm that they are of C-class.
 
 The symmetry algebra $\fs \cong \fsl_3$ of $\cE$ given by \eqref{E:5-ex} is spanned by:
  \[
 \partial_t, \quad \partial_u, \quad t\partial_t, \quad u\partial_t, \quad t\partial_u, \quad u\partial_u, \quad t^2\partial_t + tu\partial_u, \quad tu\partial_t + u^2\partial_u.
 \]
 This is a homogeneous structure and (the restriction of the prolongation of) $\fs$ is infinitesimally transitive on $\cE$.  Fixing the point $o = \{ t = u = u_1 = u_3 = u_4 = 0,\, u_2 = 0 \} \in \cE$, let us define an alternative basis:
 \begin{align*}
 X &= \partial_t + t \partial_{u}, \quad
 H = -2(t \partial_t + 2 u \partial_{u}), \quad 
 Y = 2(u-t^2) \partial_t - 2t u \partial_{u}, \\
 T_4 &= \frac{1}{2} \partial_u, \quad T_2 = -\partial_t + t\partial_u, \quad T_0 = -3t\partial_t, \\
  T_{-2} &= -2(t^2+u)\partial_t - 2tu\partial_u, \quad T_{-4} = -2tu\partial_t - 2u^2\partial_u,
 \end{align*}
 This basis is adapted to $o$: 
 \begin{itemize}
 \item the isotropy is $\fs^0 = \tspan\{ H, Y \}$.
 \item the line field $E = \tspan\{ \partial_t + u_1 \partial_{u} + ... + u_4 \partial_{u_3} + u_5 \partial_{u_4} \}$ on $\cE$ has $E|_o = \tspan\{ X|_o \}$.  Moreover, $\{ X,H,Y \}$ is a standard $\fsl_2$-triple.
 \item the line field $F = \tspan\{ \partial_{u_4} \}$ on $\cE$ has $F|_o = \tspan\{ T_{-4}|_o \}$.
 \item The elements $X$ and $T_{-4}$ have filtration degree $-1$ and this induces a filtration on $T_o\cE$.
 \end{itemize}
 (Again, we are referring to the restrictions of prolongations of the vector fields above.)
 The element $H$ was used to decompose $\fs$ into weight spaces.  Here, $T_{2i}$ has $H$-weight $2i$, and these span an $\fsl_2$-irrep isomorphic to $V_4$.  Alternatively, we can view this in terms of $3 \times 3$ trace-free matrices.  The map sending $a_2 X + a_0 H + a_{-2} Y + \sum_{i=-2}^2 b_{2i} T_{2i}$ to
 \begin{align} \label{E:sl3-split}
 \begin{psmallmatrix}
 2 a_0 & \sqrt{2} a_2 & 0\\
 \sqrt{2} a_{-2} & 0 & \sqrt{2} a_2\\
 0 & \sqrt{2} a_{-2} & -2 a_0
 \end{psmallmatrix}
 + 
 \begin{psmallmatrix}
 b_0 & -\sqrt{2}\, b_2 & b_4\\
 \sqrt{2}\, b_{-2} & -2b_0 & \sqrt{2}\, b_2\\
 b_{-4} & -\sqrt{2}\, b_{-2} & b_0
 \end{psmallmatrix}.
 \end{align}
  is a Lie algebra isomorphism $\mathfrak{s} \to \mathfrak{sl}_3$.  In summary, we have $\fs \cong \fsl_2 \op V_4$ as $\fsl_2$-modules, and this is equipped with the filtration induced from above, e.g.\ 
 $\begin{psmallmatrix}
 0 & 1 & 0\\
 0 & 0 & 1\\
 0 & 0 & 0
 \end{psmallmatrix}$ and
 $\begin{psmallmatrix}
 0 & 0 & 0\\
 0 & 0 & 0\\
 1 & 0 & 0
 \end{psmallmatrix} \mod \fs^0$ have filtration degree $-1$.
 
 The decomposition $\fsl_3 \cong \fsl_2 \op V_4$ is in fact induced by a {\em principal} $\fsl_2$ subalgebra (all of which are conjugate in $\fs$).  Similar decompositions exist for $C_2 \cong \fsp_4$ (arising from the symmetries of \eqref{E:7-ex}) and $G_2$, so it will be useful to formulate this in a uniform way.  Let $\fs$ be a rank two complex simple Lie algebra.  Fix a Cartan subalgebra $\fh$, root system $\Delta$, and a simple root system $\alpha_1,\alpha_2 \in \fh^*$.  Let $\{ h_i, e_i, f_i \}_{i=1}^2$ be standard Chevalley generators, where $e_i$ and $f_i$ are root vectors for $\alpha_i$ and $-\alpha_i$ respectively.  Let $Z_1,Z_2 \in \fh$ be the dual basis to $\alpha_1,\alpha_2$.  We use the Bourbaki ordering, so that the Cartan matrices $c_{ij} = \langle \alpha_i, \alpha_j^\vee \rangle$ for $A_2,C_2, G_2$ are:
 \[
 \begin{psmallmatrix}
 2 & -1\\
 -1 & 2
 \end{psmallmatrix}, \quad
 \begin{psmallmatrix}
 2 & -1\\
 -2 & 2
 \end{psmallmatrix}, \quad
 \begin{psmallmatrix}
 2 & -1\\
 -3 & 2
 \end{psmallmatrix}.
 \]
 Define a principal $\fsl_2$-subalgebra via the standard $\fsl_2$-triple:
 \[
 H = 2(Z_1+Z_2), \quad X = e_1 + e_2, \quad Y = 
 \begin{cases} 
 \frac{2}{3} f_1 + \frac{1}{3} f_2, & \fs = A_2;\\
 f_1 + f_2, & \fs = C_2;\\
 2 f_1 + 3 f_2, & \fs = G_2.
 \end{cases}
 \]
 The element $H$ decomposes $\fs$ into weight spaces, e.g.\ the root space with root $k\alpha_1 + \ell\alpha_2$ has weight $2(k+\ell)$.  We apply the raising operator $X$ to the lowest root space to get the irreducible summand $V_n$.  Indeed, the $H$-weight of the lowest (or highest) roots and dimension counting yields the $\fsl_2$-decomposition $\fs = \fsl_2 \op V_n$, where $n=4$ for $A_2$, $n=6$ for $C_2$, and $n=10$ for $G_2$.  Note that the sum of root spaces $\fs_{-\alpha_1} \op \fs_{-\alpha_2}$ has $H$-weight $+2$, and is decomposed into a line lying in the $\fsl_2$ and a line lying in $V_n$.  Filtration degrees are indicated in Figure \ref{F:G2-filtration} for the $G_2$ case.  (The $A_2$ and $C_2$ cases are similar.)
 \begin{center}
 \begin{figure}[h]
 \includegraphics[width=6cm]{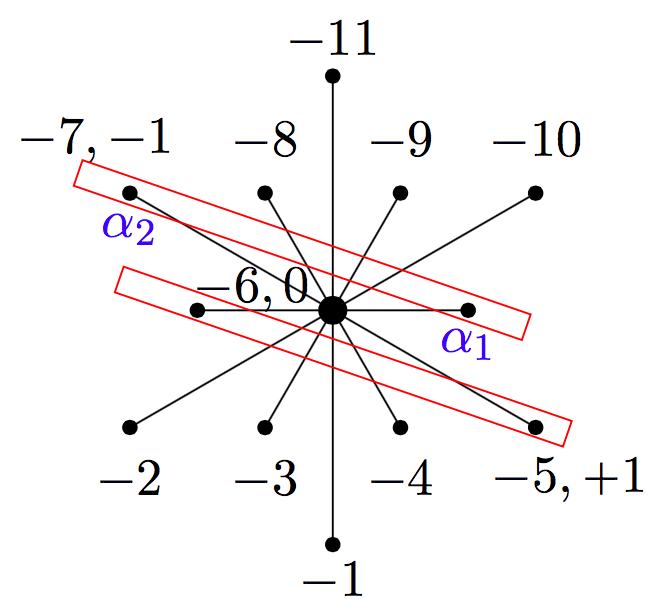}
 \caption{Filtration degrees associated with the $G_2$-model}
 \label{F:G2-filtration}
 \end{figure}
 \end{center}

 We now use all of this to describe the curvature of the 
 associated canonical Cartan geometry.  Recall from \S 1.5.15 and 1.5.16 of \cite{CS2009} that a homogeneous Cartan geometry $(\cG \to \cE, \omega)$ of type $(G,P)$ over a homogeneous base manifold $\cE \cong S / S^0$ is completely determined by a linear map $\alpha: \fs \to \fg$ that: (i) restricts to the derivative of the natural inclusion $\iota : S^0 \to P$ on $\fs^0$, (ii) is $S^0$-equivariant, i.e. $\Ad_{\iota(s)} \circ \alpha = \alpha \circ \Ad_s$ for $s \in S^0$, and (iii) induces a vector space isomorphism $\fs / \fs^0 \cong \fg / \fp$.  Letting $\tilde\kappa(x,y) = \alpha[x,y] - [\alpha(x),\alpha(y)]$, the curvature corresponds to $\kappa \in \bigwedge^2(\fg/\fp)^*\otimes \fg$ given by
 \begin{align} \label{E:hom-kappa}
 \kappa(u,v) = \tilde\kappa(\alpha^{-1}(u),\alpha^{-1}(v)).
 \end{align}

 Given the $\fsl_2$-decomposition $\fs = \fsl_2 \op V_n$, define $\alpha : \fs \inj \fg = \fgl_2 \op V_n = \fq \oplus \fa$ via the natural inclusion.  This satisfies the required conditions above,  but is moreover $\fsl_2$-equivariant.  This immediately implies that $\kappa$ given in \eqref{E:hom-kappa} vanishes upon insertion of $\sfX \,\mod \fp$, i.e. it is of the form $\begin{psmallmatrix} 0\\ \kappa_2 \end{psmallmatrix}$.
 
 We first check that $\kappa$ is normal.  Since $\alpha$ is $\fsl_2$-equivariant, then $\kappa$ can be viewed as an $\fsl_2$-{\em invariant} element of the $\fsl_2$-module $\bigwedge^2 (\fg/\fq)^* \otimes \fg \cong C^2(\fa,\fg)$.  From \eqref{E:hds}, it suffices to examine $\partial^*_\fa$ on this space.  From Lemma \ref{L:q-valued}, $C^2(\fa,\fq) \subset \ker(\partial^*_\fa)$ and $\im(\partial^*_\fa) \subset C^1(\fa,\fq)$.  As $\fsl_2$-modules, $C^1(\fa,\fq) \cong V_n \otimes (V_2 \op V_0) \cong V_{n+2} \oplus 2 V_n \oplus V_{n-2}$, which contains no trivial summands for $n \geq 3$.  By $\fsl_2$-equivariance of $\partial^*_\fa$, we conclude that $\partial^*\kappa = 0$, i.e.\ $\kappa$ is normal.
 
 A simple check using root diagrams shows that for all three cases $\kappa$ is regular and satisfies the stronger regularity condition from Remark \ref{R:strong-reg} in the $A_2$ and $C_2$ cases.  Thus, in these cases we have constructed the curvature of the canonical Cartan connection of an ODE, so these ODE are indeed of C-class.  Note that in the $A_2$ case, there is a unique trivial summand appearing in $\bigwedge^2 V_4 \otimes V_4$, so $\kappa$ necessarily lies here.  In the $C_2$ case, there are two trivial summands: one occurs in $\bigwedge^2 V_6 \otimes V_6$ and the other occurs inside $\bigwedge^2 V_6 \otimes \fsl_2$.  A direct computation shows that $\kappa$ lies in their sum, but not entirely in one summand or the other.
 
 In the $G_2$ case, $\kappa$ is not strongly regular.  The root spaces $\fs_{\alpha_1 + \alpha_2}$ and $\fs_{2\alpha_1 + \alpha_2}$ have filtration degrees $-8$ and $-9$ respectively, but these insert into $\tilde\kappa$ to produce a nontrivial element of $\fs_{3\alpha_1 + 2\alpha_2}$, which has degree $-11$.  Consequently, no corresponding $G_2$-invariant 11th order ODE exists.  We have constructed a non-ODE $G_2$-invariant filtered $G_0$-structure (with symbol algebra $\fm$).  Passing to the leaf space of the foliation by $E$, we obtain a $G_2$-invariant $\tGL_2$-structure on an 11-manifold.
 
\section{Wilczynski--flatness and the main result} \label{S:main} 

As we have observed in the end of \S \ref{S:Tanaka}, we can associate
a canonical normal Cartan geometry to any scalar ODE of order at least $4$
and each system of ODE of order at least $3$. Using the facts on the
normalization condition derived in \S \ref{S:cd-norm}, we can now
express the Wilczynski invariants in terms of the curvature $\kappa$ of this Cartan geometry. We can then prove our main result that in the case of vanishing Wilczynski invariants, the normal Cartan geometry descends to the space of solutions, thus exhibiting Wilczynski--flat equations as forming a C-class.

\subsection{Wilczynski invariants} \label{S:Wilc}

 Normality implies that $\kappa$ takes values in the subspace $\ker(\partial^*) \subset L(\bigwedge^2(\fg/\fp),\fg)$.  The element $\sfX\in\fg_-$ spans a one--dimensional $P$--invariant subspace in $\fg/\fp$.  From Definition \ref{D:C-class}, the C-class property is confirmed if $\kappa$ takes values in the $P$-submodule
 \begin{align} \label{E:bbE}
 \bbE := \{ \phi \in \ker(\partial^*) \subset C^2_{\hor}(\fg,\fg) : i_\sfX \phi = 0 \}.
 \end{align}
 In terms of the vector notation introduced in \S \ref{S:Tanaka}, this corresponds to vectors with vanishing top component. 

Composing the natural surjection $\ker(\partial^*)\to
 \ker(\partial^*)/\im(\partial^*)$ with $\kappa$, we obtain the \textit{essential curvature} $\kappa_e$ of the geometry, which is shown to be a fundamental invariant in Proposition 4.6 of \cite{Cap2016}.  From Corollary \ref{C:red}, we know that this quotient
 representation is completely reducible, which shows that $\kappa_e$ is a much
 simpler geometric object than $\kappa$.

 To describe the relation to the generalized Wilczynski invariants, we
  need some preparation. Inserting the vector $\sfX$ into the curvature function
  $\kappa$, we obtain a function with values in $L(\fg/\fp,\fg)$. By skew symmetry of
  $\kappa$, the values vanish upon insertion of $\sfX$ and thus descend to
  $L(\fg/\fq,\fg)$ and one can further project to $L(\fg/\fq,\fg/\fq)$. Now
  $\fg=\fq\oplus\fa$ and decomposing $L(\fg,\fg)$ accordingly, the result of that 
  projection can be identified with the component of $\kappa(\sfX,\_)$ in
  $L(\fa,\fa)$.

 Now $\fa$ is a representation of $\fq$, so we get $\fq\subset L(\fa,\fa)$. In
 particular $\sfY$ and, more generally, $(\sfY,A)$ for any $A\in\fgl_m$,
   can be viewed as an element of $L(\fa,\fa)$. Forming powers of the map induced by
   $\sfY$, we also get $(\sfY^k,A)\in L(\fa,\fa)$ for $k=2,\dots,n$ and
   $A\in\fgl_m$. The maps $\sfY^k$ for $k=1,\dots,n$ are linearly independent, and we
   will show that the component of $\kappa(\sfX,\_)$ in $L(\fa,\fa)$ actually has
   values in the subspace spanned by $\fq$ and the maps $(\sfY^k,A)$. Hence for each
   $k=1,\dots, n$, we get a well defined component of the curvature function
   determined by $\omega^{\sfX}\otimes\sfY^k$ with values in $\fgl_m$.  We will also
 show that the relevant components survive the projection to
 $\ker(\partial^*)/\im(\partial^*)$ so they are also components of the essential
 curvature function.

\begin{theorem}\label{T:Wilc}
 For a scalar ODE of order at least $4$ and systems of ODEs of order at least $3$, the generalized Wilczynski invariants from Definition \ref{D:Wilc} are
   equivalently encoded as components of the essential curvature function of the
   associated canonical Cartan geometry:
   \begin{itemize}
   \item $\omega^\sfX\otimes \sfY^k$ for $k=2,\dots,n$ in case of a scalar ODE or systems of ODEs of order $n+1$;
   \item additionally, the trace-free part of the component corresponding to $\omega^\sfX\otimes \sfY$ in case of a system of ODEs.
   \end{itemize}

 Vanishing of all generalized Wilczynski invariants is equivalent to
    the fact that the curvature function $\kappa$ has values in the sum
    $\bbE+\im(\partial^*)^1$, where the superscript means (filtration) homogeneity
    $\geq 1$.
\end{theorem}

\begin{proof}
  The key ingredient for this result is the description of Wilczynski
  invariants in terms of a partial connection form from
  \cite{Dou2008}. Given the manifold $\cE$ describing the equation and
  the subbundle $E\subset T\cE$ from \S \ref{S:geo-str}, we can form
  the quotient bundle $N:=T\cE/E$. Given sections $\xi$ of $E$ and $s$
  of $N$, we can choose $\eta\in\fX(\cE)$ which projects onto $s$, and
  project the Lie bracket $[\xi,\eta]$ to $N$. One immediately
  verifies that this gives rise to a well defined bilinear operation
  $D: \Gamma(E)\times\Gamma(N)\to\Gamma(N)$ which defines
  a partial connection, i.e.~it is linear over smooth functions in the
  first variable and satisfies a Leibniz rule in the second
  variable. Hence, we write it as $(\xi,s)\mapsto D_\xi s$.

  Now the canonical Cartan geometry gives rise to a specific
  description of this operation. Denoting by $p:\cG\to\cE$ the Cartan
  bundle and by $\omega$ the Cartan connection, we get
  $T\cE=\cG\times_P(\fg/\fp)$, with $E\subset T\cE$ corresponding to
  the submodule in $\fg/\fp$ spanned by $\sfX+\fp$.  This module is $\fq / \fp$, so $N = \cG\times_P(\fg/\fq)$.  Otherwise put, the bundle $\cG\to\cE$ defines
  a reduction to the structure group $P$ of the (frame bundle of the)
  vector bundle $N\to\cE$, and we can
  describe the partial connection in terms of this reduction.

    First, given a section $s\in\Gamma(N)$ and a lift
    $\eta\in\fX(\cE)$ as above, we can further lift $\eta$ to a
    $P$--invariant vector field $\tilde\eta\in\fX(\cG)$. Then by
    construction, the $P$--equivariant function $\cG\to\fg/\fq$
    corresponding to $s$ is given by $\omega(\tilde\eta)+\fq$. To
    describe the partial connection, take $\xi\in\Gamma(E)$ and choose
    a $P$--invariant lift $\tilde\xi\in\fX(\cG)$. Then the Lie bracket
    $[\tilde\xi,\tilde\eta]$ is a $P$--invariant lift of $[\xi,\eta]$,
    so $\omega([\tilde\xi,\tilde\eta])+\fq$ is the $P$--equivariant
    function corresponding to the projection of $[\xi,\eta]$ and thus
    to $D_\xi s$.

 Since $\omega(\tilde\xi)$ is $\fq$-valued, then by \eqref{E:kappa}, we have
\begin{equation}\label{E:part-conn}
 \omega([\tilde\xi,\tilde\eta]) + \fq = - \kappa(\omega(\tilde\xi),\omega(\tilde\eta)) + \tilde\xi \cdot \omega(\tilde\eta) + [\omega(\tilde\xi),\omega(\tilde\eta)] + \fq.
\end{equation}
 Let us write $f:=\omega(\tilde\eta)+\fq$ for the function
 corresponding to $s$.  The expression
\begin{equation}\label{E:conn-form}
\tau(\tilde\xi)(A+\fq)=-\kappa(\omega(\tilde\xi),A)+[\omega(\tilde\xi),A]+\fq
\end{equation}
 is well-defined since $\omega(\tilde\xi)$ has values
in $\fq$.  Thus we get a partially-defined one--form on $\cG$ with values in
$L(\fg/\fq,\fg/\fq)$ (which is only defined on tangent vectors
projecting to $E\subset T\cE$). In terms of this form, the right hand
side of \eqref{E:part-conn} reads as $\tilde\xi\cdot
f+\tau(\tilde\xi)(f)$, so $\tau$ is exactly the (partial) connection
form for $D$ on $\cG$.

  Now an interpretation of the Wilczynski invariants from \cite{Dou2008} is based on
  a proof that structure group of $N$ can be reduced to $P$ in such a way that one
  obtains a connection form for $D$ that satisfies a normalization condition. This
  condition is that its values lie in the linear subspace of
  $L(\fg/\fq,\fg/\fq)\cong\fa^*\otimes\fa$ spanned by $\fq$ and the
    maps of the form $(\sfY^2,A_2),\dots,(\sfY^n,A_n)$ with $A_i\in\fgl_m$. It is
  then shown in \cite{Dou2008} that the components as listed in the
    Theorem exactly encode the Wilczynski invariants.

  To see that $\tau$ satisfies this normalization condition, observe first that
  $\omega(\tilde\xi)$ has values in $\mathbb R\sfX\subset \fq$, so it suffices to deal with the case that
    $\omega(\tilde\xi)=\sfX$. Also, the bracket--term in \eqref{E:conn-form} gives a
  contribution in $\fq\subset\fa^*\otimes\fa$.  Using the vector notation
  $\binom{\kappa_1}{\kappa_2}$ from \S \ref{S:codiff-explicit} for (the values of)
  $\kappa$, we get $\kappa(\sfX,\_)=\kappa_1$. The values of $\kappa_1$
  lie in $C^1(\fa,\fg)$, which as we know admits a $\fq$--invariant decomposition as
  $C^1(\fa,\fq)\oplus C^1(\fa,\fa)$. Since in \eqref{E:conn-form} we work modulo
  $\fq$, the component of $\kappa$ showing up there is a multiple of the component
  $\kappa_1^{\fa}$ of $\kappa_1$ in $C^1(\fa,\fa)$. But by Proposition \ref{P:hds},
  normality of $\kappa$ implies that $\kappa_1$ has values in
  $\ker(\partial^*_{\fa})$, while $\sfY\cdot\kappa_1$ has values in
  $\im(\partial^*_{\fa})$. By Lemma \ref{L:q-valued}, this means that
  $\kappa_1^{\fa}\in\fq^\perp\subset\fa^*\otimes\fa$ and that
  $\sfY\cdot\kappa_1^{\fa}=0$. This exactly means that the values of $\kappa_1^{\fa}$
  lie in the sum of all those lowest weight spaces of the $\fsl_2$--representation
  $\fa^*\otimes\fa$, which are perpendicular to the submodule $\fq$.
    These lowest weight spaces are spanned by the maps
    $(\sfY,A_1),(\sfY^2,A_2),\dots,(\sfY^n,A_n)$ with $A_1\in\fsl_m$ and
    $A_k\in\fgl_m$ for $k=2,\dots,n$.

  Thus we conclude that $\tau$ satisfies the normalization condition and that the
  Wilczynski invariants are equivalently encoded by the class of
    $\kappa(\sfX,\_)$ modulo $C^1(\fa,\fq)$. In particular, this class and thus
    $\kappa_1^{\fa}$ vanishes identically in the Wilczynski--flat case.

 Having all that in hand, the claims in the theorem now follow from two simple
 observations. On the one hand, Proposition \ref{P:hds} and Lemma \ref{L:q-valued}
 show that the restriction of the $P$--equivariant map
 $\phi=\binom{\phi_1}{\phi_2}\mapsto \phi_1^{\fa}$ to $\ker(\partial^*)\subset
 C^2_{\hor}(\fg,\fg)$ vanishes on the subspace $\im(\partial^*)$. Thus it factorizes
 to the quotient $\ker(\partial^*)/\im(\partial^*)$, which shows that
 $\kappa_1^{\fa}$ and thus the components encoding the Wilczynski
   invariants are components of the essential curvature function $\kappa_e$. This
   also shows that if $\kappa$ has values in $\mathbb E+\im(\partial^*)^1$, then
   $\kappa_e$ vanishes upon insertion of $\sfX$ and hence all Wilczynski invariants
   vanish.

 On the other hand, suppose that we start from a Wilczynski--flat equation, so
 $\kappa = \binom{\kappa_1}{\kappa_2} \in \ker(\partial^*)$ has the property that
 $\kappa_1^{\fa}=0$. Then by Lemma \ref{L:q-valued}, $\kappa_1=\kappa_1^\fq \in
 \im(\partial^*_\fa)$, so we can take an element $\psi_1\in C^2(\fa, \fg)$ such that
 $\partial^*_\fa\psi_1=-\kappa_1$.  Compatibility with homogeneities shows that we
 may assume that $\psi_1$ is homogeneous of degree $\geq 0$.  But then by Proposition
 \ref{P:hds}, $\binom{\kappa_1}{\kappa_2}-\partial^*\binom{\psi_1}{0}$ has vanishing
 top--component and thus lies in $\bbE$.  Hence,
 $\binom{\kappa_1}{\kappa_2}\in\bbE+\im(\partial^*)^1$, which completes the proof.
\end{proof}

 \subsection{The covariant exterior derivative}
 \label{S:d-omega}

 Let $(\cG \to M, \omega)$ be a regular Cartan geometry of type
 $(\fg,P)$. Following \cite{Cap2016}, we consider the operator
 $d^\omega:\Omega^k(\cG,\fg)\to\Omega^{k+1}(\cG,\fg)$ defined by
 \[
 (d^\omega \varphi)(\xi_0,...,\xi_k) = d\varphi(\xi_0,...,\xi_k) +
 \textstyle\sum_{i=0}^k (-1)^i [\omega(\xi_i),
 \varphi(\xi_0,...,\widehat{\xi_i}, ..., \xi_k)],
 \]
 where $\xi_j \in \fX(\cG)$ for all $j$. By Proposition 4.2 of
 \cite{Cap2016}, if $\varphi$ is horizontal and $P$--equivariant, then
 so is $d^\omega\varphi$. Moreover, $d^\omega$ is compatible with the
 natural notion of homogeneity for $\fg$--valued differential forms,
 and the curvature $K$ of $\omega$ satisfies the Bianchi identity
 $d^\omega K = 0$. 

 \subsection{Wilczynski-flat ODE are of C-class}

 Now we are ready to prove our main result.
 
 \begin{theorem} \label{T:Wilc-C} Any Wilczynski-flat ODE
   \eqref{E:ODE} with $m=1$, $n \geq 3$ or $m \geq 2$, $n \geq 2$ is
   of C-class.
 \end{theorem}
 
 \begin{proof} Let $(\cG \to \cE, \omega)$ be the regular, normal
   Cartan geometry of type $(\fg,P)$ associated to \eqref{E:ODE} as in
   Theorem \ref{T:G0str-Cartan}. We have to show that for a
   Wilczynski--flat ODE, the curvature function $\kappa$ has values in
   the module $\bbE$ defined in \eqref{E:bbE}. Generalizing the
   relation between the curvature $K$ and the curvature function
   $\kappa$, horizontal $\fg$--valued $k$--forms on $\cG$ can be
   naturally identified with smooth functions $\cG\to
   L(\bigwedge^k(\fg/\fp),\fg)$.  The natural notions of
   $P$--equivariance in the two pictures correspond to each other, see
   Theorem 4.4 of \cite{Cap2016}. For the current proof, it will be
   helpful to switch between forms and equivariant functions freely,
   so we will express the fact that $\kappa$ has values in $\bbE$ as
   ``$K$ lies in $\bbE$''. In these terms, composing functions with
   $\partial^*$ defines a tensorial operator
   $\Omega^k_{\hor}(\cG,\fg)\to\Omega^{k-1}_{\hor}(\cG,\fg)$ for each
   $k$, and we also denote this operator by $\partial^*$. By
   construction, $\partial^*$ maps $P$--equivariant forms to
   $P$--equivariant forms. In this language, normality can be simply
   expressed as $\partial^*K=0$.

   By Theorem \ref{T:Wilc}, Wilczynski--flatness implies that $K$ has
   values in $\bbE+\im(\partial^*)^1$. Passing to equivariant
   functions, applying Lemma 4.7 of \cite{Cap2016}, and passing back
   to differential forms, we conclude that $K=K_1+K_2$ for
   $P$--equivariant forms $K_1,K_2\in\Omega^2_{\hor}(\cG,\fg)$ such
   that $K_1$ has values in $\bbE$ and $K_2$ has values in
   $\im(\partial^*)^1$. Now we prove the theorem in a recursive way by
   showing that for any $\ell\geq 1$, from a decomposition $K=K_1+K_2$
   such that $K_1$ has values in $\bbE$ and $K_2$ has values in
   $\im(\partial^*)^\ell$, we can always obtain a decomposition
   $K=\tilde K_1+\tilde K_2$, for which $\tilde K_1$ again has values
   in $\bbE$ but $\tilde K_2$ has values in
   $\im(\partial^*)^{\ell+1}$. Since $\im(\partial^*)^r=\{0\}$ for
   sufficiently large $r$, this implies the result.
 
   So let us assume that $K=K_1+K_2$ as above with $K_2$ having values
   in $\im(\partial^*)^\ell$ for some $\ell\geq 1$. We first claim
   that for a $P$--equivariant form $\varphi \in
   \Omega^2_{\hor}(\cG,\fg)$, which has values in $\bbE$, also
   $\partial^*d^\omega\varphi$ has values in $\bbE$. In terms of the
   description of $\partial^*$ from Proposition \ref{P:hds}, lying in
   $\bbE$ means that the top component of the right hand side of
   \eqref{E:hds} has to vanish. Denoting by $\hat X\in\fX(\cG)$ the
   vector field characterized by $\omega(\hat X)=\sfX$, we thus have
   to show that (the equivariant function corresponding to) $i_{\hat
     X}d^\omega\varphi$ has values in
   $\ker(\partial^*_\fa)\subset C^2(\fa,\fg)$.

    The assumption on $\varphi$ implies $i_{\hat X}\varphi=0$, so for
   vector fields $\xi,\eta\in\fX(\cG)$, we get
 \begin{equation}\label{E:ixd}
  (i_{\hat X}
  d^\omega\varphi)(\xi,\eta)=d\varphi(\hat X,\xi,\eta)+[X,\varphi(\xi,\eta)].
\end{equation} 
 Using $i_{\hat X}\varphi=0$ once more, we get
\begin{equation}\label{E:ixd2}
\begin{aligned}
  d\varphi&(\hat X,\xi,\eta)=\hat
  X\cdot\varphi(\xi,\eta)-\varphi([\hat
  X,\xi],\eta)-\varphi(\xi,[\hat X,\eta])=\\
  & \hat X\cdot f(\omega(\xi),\omega(\eta))-f(\omega([\hat
  X,\xi]),\omega(\eta))-f(\omega(\xi),\omega([\hat X,\eta])).
\end{aligned}
\end{equation}
 Here $f$ denotes the equivariant function corresponding to $\varphi$,
 which takes values in $\bbE\subset\ker(\partial^*)$.  
 By Proposition \ref{P:hds}, $f$ in fact has values 
 in $\ker(\partial^*_{\fa})\subset C^2(\fa,\fg)$.
 Since $\omega(\hat X)$ is constant, then by \eqref{E:kappa}, 
 \begin{equation}\label{E:ixd3}
  \omega([\hat X,\xi]) =-K(\hat X,\xi)+[X,\omega(\xi)]+\hat X\cdot\omega(\xi), 
 \end{equation}
 and likewise for $\omega([\hat X,\eta])$.

 Now by assumption $K=K_1+K_2$ and $i_{\hat X}K_1=0$, 
 so $K(\hat X,\xi)=\kappa_2(\sfX,\omega(\xi))$ and $\kappa_2$ 
 has values in $\im(\partial^*)$. By Proposition \ref{P:hds}, this means that
 $\kappa_2(\sfX, \_)$ has values in $\im(\partial^*_{\fa})$,
 which by part (3) of Lemma \ref{L:q-valued} is contained in $\fq$. In
 particular, terms of the form $K(\hat X,\xi)$ insert trivially into
 $f$, so these do not contribute. On the other hand, the contribution
 to \eqref{E:ixd2} resulting from the last term in \eqref{E:ixd3} is 
 $$
 -f(\hat X\cdot\omega(\xi),\omega(\eta))-f(\omega(\xi),\hat X\cdot\omega(\eta)). 
 $$ 
 This adds up with the first term in the right hand side of
 \eqref{E:ixd2} to $(\hat X\cdot f)(\omega(\xi),\omega(\eta))$. 
 Since $f$ has values in $\bbE$, the derivative $\hat X\cdot f$
 has the same property. Now inserting the last remaining term in the
 right hand side of \eqref{E:ixd3} into the right hand side of
 \eqref{E:ixd2}, we obtain 
 $$
 -f([X,\omega(\xi)],\omega(\eta))-f(\omega(\xi),[X,\omega(\eta)]). 
 $$
 Viewing $f$ as a function with values in 
 $\ker(\partial^*_{\fa})\subset C^2(\fa,\fg)$ as above, we can write
 the sum of these terms with the last term in the right hand side of
 \eqref{E:ixd} as $(\rho_{\sfX}\circ f)(\omega(\xi),\omega(\eta))$. 
 Here $\rho_{\sfX}$ denotes the natural
 action of $\sfX\in\fq$ on $C^2(\fa,\fg)$. But then $\fq$--equivariance
 of $\partial^*_{\fa}$ as proved in part (1) of Lemma \ref{L:q-valued}
 shows that also this function has values in $\ker(\partial^*_{\fa})$,
 which completes the proof of the claim.

Returning to our decomposition $K=K_1+K_2$, we now use the Bianchi
identity $d^\omega K=0$ to get $\partial^*d^\omega
K_2=-\partial^*d^\omega K_1$, so by the claim, this has values in
$\bbE$. Now consider the maps $\partial_{\fg_-}$ and
$\underline{\partial^*}$ defined on the spaces $C^k(\fg_-,\fg)$ as in
the proof of Proposition \ref{P:codiff}, and the algebraic Laplacian
$\Box:=\partial_{\fg_-}\circ
\underline{\partial^*}+\underline{\partial^*}\circ \partial_{\fg_-}$,
which preserves degrees and homogeneity. This clearly can be
restricted to an endomorphism of $\im(\underline{\partial^*})\subset
C^2(\fg_-,\fg)_\ell$ (on which it coincides with
$\underline{\partial^*}\circ \partial_{\fg_-}$), and in the proof of
Proposition 3.10 of \cite{Cap2016} it is shown that this restriction
is bijective. By the Cayley--Hamilton theorem, there is a polynomial
$p_\ell \in \bbR[x]$ such that $p_\ell(\Box)$ is inverse to $\Box$ on
$\im(\underline{\partial^*})_\ell$.  Now $p_\ell(\partial^* d^\omega)$
is a well-defined operator on the space of $P$--equivariant forms in
$\Omega^2_{\hor}(\cG,\fg)$, which preserves homogeneities.

Applying our claim once more, we see that $p_\ell(\partial^*
d^\omega)\partial^* d^\omega K_2$ has values in $\bbE$ and by
construction is still homogeneous of degree $\geq\ell$. Thus, $\tilde
K_1=K_1+p_\ell(\partial^*d^\omega)\partial^* d^\omega K_2$ has values
in $\bbE$, while $\tilde K_2:=K_2-p_\ell(\partial^*
d^\omega)\partial^* d^\omega K_2$ has values in $\im(\partial^*)$ and
is homogeneous of degree $\geq\ell$. To verify that $K=\tilde
K_1+\tilde K_2$ is the desired decomposition, it suffices to show that
the homogeneous component of degree $\ell$ of (the equivariant
function corresponding to) $\tilde K_2$ vanishes identically.

By part (3) of Theorem 4.4 of \cite{Cap2016}, for a $P$--equivariant
form $\varphi\in\Omega^2(\cG,\fg)$ which is homogeneous of degree
$\geq\ell$ and corresponds to the equivariant function $f$, the
homogeneous component of degree $\ell$ of the equivariant function
corresponding to $\partial^*d^\omega\varphi$ is given by
$\underline{\partial^*}\circ \partial_{\fg_-} \circ\tgr_\ell\circ f$. Applying
this iteratively starting with the function $\kappa_2$ corresponding
to $K_2$, we remain in the realm of functions having values in
$\im(\partial^*)^\ell$. Thus we iteratively conclude that the
homogeneous component of degree $\ell$ of the function corresponding
to $p_\ell(\partial^* d^\omega)\partial^* d^\omega K_2$ coincides with
$p_\ell(\Box)\circ\Box\circ\tgr_\ell\circ\kappa_2=\tgr_\ell\circ\kappa_2$,
which shows that $\tilde K_2$ has vanishing homogeneous component of
degree $\ell$.
\end{proof}

 \begin{example} \label{EX:circles}
 Let $\bu = (u^1,...,u^m)$.  The equations for circles in an $(m+1)$-dimensional Euclidean space are given by
 \[
 \bu''' = 3 \bu'' \frac{\langle \bu', \bu'' \rangle}{1 + \langle \bu', \bu' \rangle}.
 \]
 This ODE system is conformally invariant, and it has been verified that the Wilczynski invariants vanish \cite[Prop.2]{Med2011}.  By our Theorem \ref{T:Wilc-C}, the system is of C-class.  For $m=1$, the equation is (contact) trivializable, but for $m \geq 2$ the system is not (point) trivializable.
 \end{example}

Since the Wilczynski invariants for linear equations are invariants,
it follows that an ODE with trivializable linearizations is Wilczynski--flat. Hence we obtain
\begin{cor}\label{C:main}  Let $n \geq 2$, $m \geq 1$, with $(n,m) \neq (2,1)$.
  Any ODE \eqref{E:ODE} for which the linearization around any solution is
  trivializable, is of C-class.
\end{cor}

\begin{example} \label{EX:n-submax} The ODE $u^{(n+1)} - \frac{n+1}{n} \frac{(u^{(n)})^2}{u^{(n-1)}} = 0$, for $n \geq 3$, is submaximally symmetric \cite[p.206]{Olv1995} except when $n=4$ and $6$.
 At a fixed solution $u$, its linearization is the ODE for $v$ given by
 \begin{align} \label{E:n-lin}
 \ell_u[v] := v^{(n+1)} - \frac{2(n+1)}{n} a v^{(n)} -
 \frac{n+1}{n} a^2 v^{(n-1)} = 0,
 \end{align}
 where $a := \frac{u^{(n)}}{u^{(n-1)}}$.  We have $a' = \frac{a^2}{n}$ and
 hence $(\frac{1}{a})'=-\frac{1}{n}$, $(\frac{1}{a^2})'=-\frac{2}{na}$ and $(\frac{1}{a^2})'' =
 \frac{2}{n^2}$. Defining $\tilde{v} = \frac{v}{a^2}$, we have
 $\tilde{v}^{(n+1)} = \frac{1}{a^2} \ell_u[v] = 0$, so \eqref{E:n-lin}
 is trivializable, and the given ODE is of C-class by Corollary
 \ref{C:main}.  This example is included into a larger family of Wilczynski-flat (hence C-class) equations given in \cite[Example 2]{Dou2008}.
 \end{example}

 \begin{example} \label{EX:n-sys}
 Let $m \geq 2$ and $n \geq 2$.  Given $\bu = (u^1,...,u^m)$, consider
 \[
 \bu^{(n+1)} = \mathbf{f}, \qbox{where} 
 f_i = \begin{cases}
 0, & i\neq m;\\
 ((u^1)^{(n)})^2, & i=m.
 \end{cases}
 \]
 Its linearization is easily seen to be trivializable, so it is of C-class. It is not trivializable since a fundamental invariant does not vanish on it, namely $I_2$ in \cite{DM2014}.  A similar 2nd order example was given in \cite[(5.6a)]{KT2014}, which was known to be of C-class since it is torsion-free \cite{Gro2000}.
 \end{example}


\subsection{Remark: A potential alternative line of argument}
To conclude the article, let us briefly outline how the theory we have
developed could be used to obtain an alternative proof of our main
result. This line of argument is based on correspondence spaces which
are familiar in the case of parabolic geometries. It depends crucially
on the existence of a natural Segr\'e structure on the space of
solutions of a Wilczynski-flat ODE from \cite{Dou2008}, compare with
\S \ref{S:soln-space-geo-str}. As mentioned there, Segr\'e structures
are classical first order structures corresponding to $Q\subset
\tGL(\fa)$. Hence such a structure on a space $S$ comes with a
$Q$--principal bundle $\cG\to S$. The classical way to study such
structures is via the Spencer differential. As we have noted in the
proof of Proposition \ref{P:Tanaka}, this coincides with the
restriction of $\partial_{\fa}$ to a map $C^1(\fa,\fq)\to
C^2(\fa,\fa)$ and is injective. Choosing a $Q$--invariant complement
$\cN\subset C^2(\fa,\fa)$ to the image of the Spencer differential,
there is a canonical principal connection form $\tau$ on $\cG$
characterized by the fact that its torsion lies in
$\cG\times_Q\cN\subset\bigwedge^2T^*S\otimes TS$.

Usually, not too much emphasis is put on the actual choice of $\cN$,
but in the case of Segr\'e structures, this is a surprisingly subtle
issue. Analyzing $\partial_\fa : \fa^* \otimes \fq \to \bigwedge^2 \fa^* \otimes \fa$ in terms of representations of $\fq$, one
easily deduces that there always exist $Q$--invariant complements, but
aside from the $(m,n)=(1,3)$ case, there is always a freedom of
choice.  The larger $m$ and $n$ get, the bigger this freedom becomes,
and while there are always only finitely many free parameters
involved, their number gets arbitrarily high.

Now it turns out that the construction from \S \ref{S:codiff} can
also be used to construct uniform normalization conditions for Segr\'e
structures. Indeed, we can view $L(\bigwedge^2(\fg/\fq),\fg)$ as the
subspace in $C^2(\fg,\fg)$ consisting of all cochains vanishing upon insertion of one element of $\fq$. Similarly as in Lemma \ref{L:codiff} one shows that
this subspace is preserved by $\partial^*$ and that the restriction of
$\partial^*$ to it is $Q$--equivariant. Using a similar adjointness
result as in Proposition \ref{P:codiff}, one shows that
$\cN:=\ker(\partial^*_\fa)\subset C^2(\fa,\fa)$ is a $Q$--invariant
complement to the image of the Spencer differential. 

 Now the alternative approach for proving that Wilczynski--flat ODE
 form a C-class goes as follows. Starting with such an ODE $\cE$,
 form a local space $\cS$ of solutions. As proved in \cite{Dou2008},
 this space of solutions inherits a natural Segr\'e structure. This
 gives rise to a principal $Q$--bundle $\cG\to\cS$, which we may endow
 with the canonical principal connection $\tau\in\Omega^1(\cG,\fq)$
 for the choice $\cN:=\ker(\partial^*_\fa)$ of normalization
 condition. Taking the canonical soldering form $\theta$ on $\cG$,
 which can be viewed as having values in $\fa$, we can form
 $\theta\oplus\tau$, and this defines a Cartan connection $\omega$ of
 type $(\fg,Q)$ on $\cG$.

 The restriction of the principal action of $Q$ defines a free right
 action of $P\subset Q$ on $\cG$, and we can form the correspondence
 space, i.e.\ the space $\cC\cS:=\cG/P$ of orbits. This can be
 identified with the total space of the associated bundle
 $\cG\times_Q(Q/P)$. Of course, $\cG\to\cC\cS$ is a principal
 $P$--bundle and it is easy to verify that $\omega$ also is a Cartan
 connection of type $(\fg,P)$ on $\cG\to\cC\cS$. 

 Guided by what happens for parabolic geometries, we expect that it is possible 
to show that $\cC\cS$ is locally isomorphic
 to $\cE$, so $(\cG,\omega)$ can be locally viewed as a Cartan
 geometry of type $(\fg,P)$ over $\cE$. This isomorphism is expected to have
 the property that the underlying filtered $G_0$--structure of this
 Cartan geometry is the given structure on $\cE$. From our choice of
 normalization conditions it follows that $\omega$ is also normal (in
 the sense used in this article) as a Cartan connection of type
 $(\fg,P)$. Uniqueness of the normal Cartan geometry implies that
 $(\cG,\omega)$ is locally isomorphic to the canonical Cartan geometry
 on $\cE$ and by construction it descends to the local space $\cS$ of
 solutions, which would complete the argument.

 \section*{Acknowledgements}
 
 The first and third authors were respectively supported by projects
 P27072-N25 and M1884-N35 of the Austrian Science Fund (FWF).
 D.T. was also supported by the Troms\o{} Research Foundation.

 \end{document}

%% file: DT-macros.tex
\newcommand\diag{\operatorname{diag}}

 \newcommand\id{\mathrm{id}}
 \newcommand\tr{\mathrm{tr}}

 \newcommand\bop{\bigoplus}
 \newcommand\op{\oplus}

 \newcommand\fann{\mathfrak{ann}}

 \newcommand\ba{{\mathbf a}}

 \newcommand\bu{{\mathbf u}}

 \newcommand\im{\mathrm{im}}

 \newcommand\fa{{\mathfrak a}}

 \newcommand\fg{{\mathfrak g}}
 \newcommand\fgl{\mathfrak{gl}}
 \newcommand\fh{{\mathfrak h}}

 \newcommand\fm{{\mathfrak m}}

 \newcommand\fp{{\mathfrak p}}
 \newcommand\fq{{\mathfrak q}}
 
 \newcommand\fs{{\mathfrak s}}
 \newcommand\fsl{\mathfrak{sl}}
 
 \newcommand\fsp{\mathfrak{sp}}

 \newcommand\fz{{\mathfrak z}}

 \newcommand\fC{{\mathfrak C}}

 \newcommand\fG{{\mathfrak G}}

 \newcommand\fP{{\mathfrak P}}

 \newcommand\fX{{\mathfrak X}}

 \newcommand\cC{{\mathcal C}}
 
 \newcommand\cE{{\mathcal E}}
 
 \newcommand\cG{{\mathcal G}}

 \newcommand\cN{{\mathcal N}}

 \newcommand\cS{{\mathcal S}}

 \newcommand\cW{{\mathcal W}}

 \newcommand\sfe{\mathsf{e}}

 \newcommand\sfv{\mathsf{v}}
 
 \newcommand\sfx{\mathsf{x}}
 \newcommand\sfy{\mathsf{y}}

 \newcommand\sfH{\mathsf{H}}

 \newcommand\sfX{\mathsf{X}}
 \newcommand\sfY{\mathsf{Y}}
 \newcommand\sfZ{\mathsf{Z}}

 \newcommand\bbE{{\mathbb E}}

 \newcommand\bbP{{\mathbb P}}
 
 \newcommand\bbR{{\mathbb R}}

 \newcommand\bbZ{{\mathbb Z}}

 \newcommand\tspan{\mathrm{span}}

 \newcommand\Ben{\begin{enumerate}}
 \newcommand\Een{\end{enumerate}}
 \newcommand\Bex{\begin{example}}
 \newcommand\Eex{\end{example}}

 \def\inj{\hookrightarrow}

\def\tGL{\text{GL}}

 \newcommand\ad{{\rm ad}}
 
 \newcommand\Ad{{\rm Ad}}
 \newcommand\Aut{\text{Aut}}
 \newcommand\tgr{\mathrm{gr}}

 \def\assoc/{associative}
 \def\arb/{arbitrary}
 \def\btw/{between}
 \def\coeff/{coefficient}
 \def\cohom/{cohomology}
 \def\coord/{coordinate}
 \def\coordsys/{coordinate system}
 \def\cpt/{compact}
 \def\cred/{completely reducible}
 \def\cts/{continuous}
 \def\dga/{differential-graded algebra}
 \def\dR/{de Rham}
 \def\Euc/{Euclidean} 
 \def\grp/{group}
 \def\hom/{homomorphism}
 \def\inv/{invariant}
 \def\iso/{isomorphism}
 \def\La/{Lie algebra}
 \def\Lag/{Lagrangian Grassmannian}
 \def\LG/{Lie group}
 \def\MA/{Monge--Amp\`ere}
 \def\MC/{Maurer--Cartan}
 \def\lintr/{linear transformation} 
 \def\mfld/{manifold}
 \def\nb/{normal bundle}
 \def\nbd/{neighbourhood}
 \def\nondeg/{non-degenerate}
 \def\posdef/{positive definite}
 \def\pu/{partition of unity}
 \def\rep/{representation}
 \def\Riem/{Riemannian}
 \def\sg/{subgroup}
 \def\ss/{semi-simple}
 \def\inv/{invariant}
 \def\irr/{irreducible}
 \def\Jacid/{Jacobi identity}
 \def\li/{linearly independent}
 \def\nd/{nowhere dependent}
 \def\nz/{nowhere zero}
 \def\on/{orthonormal}
 \def\onb/{\on/ basis}
 \def\orc/{\orth/ complement}
 \def\orth/{orthogonal}
 \def\orp/{\orth/ projection}
 \def\pde/{partial differential equation}
 \def\resp/{respectively}
 \def\seq/{sequence}
 \def\std/{standard}
 \def\SW/{Stiefel-Whitney}
 \def\uc/{universal cover}
 \def\vb/{vector bundle}
 \def\vf/{vector field}
 \def\vs/{vector space}
 \def\wrt/{with respect to}
 
 \renewcommand\mod{\,{\rm mod}\ }
 \newcommand\qbox[1]{\quad\mbox{#1}\quad}
 \renewcommand\dim{{\rm dim}}

%% file: C-class.bbl
\begin{thebibliography}{1}

 \bibitem{Bry1991} R. Bryant, Two exotic holonomies in dimension four,
   path geometries, and twistor theory, Proc. Symp. Pure Math. {\bf
     53} (1991), 33--88.

 \bibitem{CD2001} D.M.J. Calderbank, T. Diemer, Differential
   invariants and curved Bernstein--Gelfand--Gelfand sequences, J.
   reine angew. Math. {\bf 537} (2001), 67--103.

 \bibitem{Cap2005} A. \v{C}ap, Correspondence spaces and twistor spaces for parabolic geometries, J. Reine Angew. Math. {\bf 582} (2005), 143--172.
 
 \bibitem{Cap2016} A. \v{C}ap, On canonical Cartan connections associated to filtered G-structures, preprint arXiv:1707.05627 (2017).
 
 \bibitem{CS2009} A. \v{C}ap, J. Slov\'ak, Parabolic Geometries I:
   Background and General Theory, American Mathematical Society, 2009.

 \bibitem{CSS2001} A. \v{C}ap, J. Slov\'ak, V. Sou\v{c}ek,
   Bernstein--Gelfand--Gelfand sequences, Ann. Math. {\bf 154} (2001),
   97--113.

 \bibitem{CSo2017} A. \v{C}ap, V. Sou\v{c}ek, Relative BGG sequences;
   II. BGG machinery and invariant operators, 
   Adv. Math. {\bf 320} (2017) 1009--1062.
   
 \bibitem{Car1938} \'E.\ Cartan, Les espaces g\'en\'eralis\'es et l'int\'egration de certaines classes d'\'equations diff\'erentielles, C.R., 1938, V.206, N.23, 1689--1693.
 
 \bibitem{Dou2001} B.\ Doubrov, Contact trivialization of ordinary differential equations, Differential geometry and its applications. Proceedings of the 8th international conference, Opava, Czech Republic, August 27--31, 2001. Math. Publ. (Opava) {\bf 3} (2001), pp. 73--84.

 \bibitem{Dou2008} B.\ Doubrov, Generalized Wilczynski invariants for nonlinear ordinary differential equations, The IMA Volumes in Mathematics and its Applications {\bf 144} (2008), 25--40.

 \bibitem{DKM1999} B.\ Doubrov, B. Komrakov, T. Morimoto, Equivalence of holonomic differential equations, Lobachevskij Journal of Mathematics {\bf 3} (1999), 39--71.
 
 \bibitem{DM2014} B.\ Doubrov, A.\ Medvedev, Fundamental invariants of systems of ODEs of higher order. Differential Geom. Appl. {\bf 35} (2014), suppl., 291--313.
 
 \bibitem{DG2012} M. Dunajski, M. Godli\'nski, $GL(2,\mathbb{R})$ structures, $G_2$ geometry and twistor theory, Q. J. Math. {\bf 63}(2012), no. 1, 101--132.

 \bibitem{DT2006} M. Dunajski, P. Tod, Paraconformal geometry of $n$-th order ODEs, and exotic holonomy in dimension four, J. Geom. Phys. 56 (2006) 1790--1809.
 
 \bibitem{Fels1995} M.E. Fels, The equivalence problem for systems of second-order ordinary differential equations, Proc. London Math. Soc. (3) {\bf 71} (1995), 221--240.
  
 \bibitem{GN2010} M. Godlinski, P. Nurowski, $\tGL(2,\mathbb{R})$ geometry of ODE's, J. Geom. Phys. {\bf 60} (2010), no. 6-8, 991--1027.
 
 \bibitem{Gro2000} D. Grossman, Torsion-free path geometries and integrable second order ODE systems, Selecta Math. {\bf 6}, no. 4 (2000), 399--442.

 \bibitem{KN1964} S.\ Kobayashi, T.\ Nagano, On filtered Lie algebras and geometric structures. I. J. Math. Mech. {\bf 13} (1964), 875--907.

 \bibitem{KN1965} S.\ Kobayashi, T.\ Nagano, On filtered Lie algebras and geometric structures. III. J. Math. Mech. {\bf 14} (1965), 679--706.

 \bibitem{KT2014} B.\ Kruglikov, D.\ The, The gap phenomenon in parabolic geometries. J. Reine Angew. Math. {\bf 2017}(723), 153--215 (2014).

 \bibitem{Kry2010} W.\ Kry\'nski, Paraconformal structures and differential equations. Differential Geom. Appl. {\bf 28} (2010), no. 5, 523--531.
 
 \bibitem{Med2011} A.\ Medvedev, Third order ODEs systems and its characteristic connections, SIGMA {\bf 7} (2011), Paper 076, 15 pp.
 
 \bibitem{Mor1993} T. Morimoto, Geometric structures on filtered manifolds, Hokkaido Math. J. {\bf 22} (1993), 263--347.

 \bibitem{Nur2009} P. Nurowski, Comment on $\operatorname{GL}(2,\mathbb{R})$ geometry of fourth-order ODEs, J. Geom. Phys. {\bf 59} (2009), no. 3, 267--278.
 
 \bibitem{Olv1995} P.J. Olver, Equivalence, Invariants and Symmetry, Cambridge University Press, 1995.
 
 \bibitem{Seashi1988} Y. Se-ashi, On differential invariants of integrable finite type linear differential equations, Hokkaido Math. J. {\bf 17} (1988), 151--195.
  
 \bibitem{Tanaka2017a} N. Tanaka, K. Kiyohara, T. Morimoto, K. Yamaguchi, Geometric Theory of Systems of Ordinary Differential Equations I, Hokkaido University Technical Report Series in Mathematics, no. 169, March 2017, 1--169. \href{http://doi.org/10.14943/81531}{http://doi.org/10.14943/81531}
 
 \bibitem{Tanaka2017b} N. Tanaka, K. Kiyohara, T. Morimoto, K. Yamaguchi, Geometric Theory of Systems of Ordinary Differential Equations II, Hokkaido University Technical Report Series in Mathematics, no. 170, March 2017. \href{http://doi.org/10.14943/81532}{http://doi.org/10.14943/81532}
 
 \bibitem{Wilc1905} E.J. Wilczynski, Projective differential geometry of curves and ruled surfaces, Leipzig, Teubner, 1905.
  
 \end{thebibliography}
